\documentclass[a4paper,12pt]{article}

\usepackage{amsthm}
\usepackage{latexsym}
\usepackage{dsfont}
\usepackage{bbm}
\usepackage{amssymb}
\usepackage{amsmath}
\usepackage{graphicx}

\numberwithin{equation}{section}

\theoremstyle{plain}
\newtheorem{Theorem}{Theorem}[section]
\newtheorem{Lemma}[Theorem]{Lemma}

\newtheorem{Prop}[Theorem]{Proposition}
\theoremstyle{remark}
\newtheorem{Rem}[Theorem]{Remark}
\theoremstyle{definition}
\newtheorem{Exa}[Theorem]{Example}

\DeclareMathOperator{\N}{\mathbb{N}}

\DeclareMathOperator{\R}{\mathbb{R}}

\DeclareMathOperator{\Prob}{\mathbb{P}}

\DeclareMathOperator{\E}{\mathbb{E}}

\DeclareMathOperator{\1}{\mathbbm{1}}

\title{Power and Exponential Moments of the Number of Visits and Related Quantities for Perturbed Random Walks}
\date{\today}

\author{Gerold Alsmeyer\footnote{Institut f\"{u}r Mathematische Statistik,
Westf\"{a}lische Wilhelms-Universit\"{a}t M\"{u}nster, 48149
M\"{u}nster, Germany, e-mail: gerolda@math.uni-muenster.de}\,,
Alexander Iksanov\footnote{ Faculty of Cybernetics, National T.\
Shevchenko University of Kiev, 01033 Kiev, Ukraine, e-mail:
iksan@univ.kiev.ua}\ \,and Matthias
Meiners\footnote{ Institut f\"{u}r Mathematische Statistik,
Westf\"{a}lische Wilhelms-Universit\"{a}t M\"{u}nster, 48149
M\"{u}nster, Germany, e-mail: mmeiners@math.uni-muenster.de}}

\begin{document}

\thispagestyle{empty}
\maketitle

\begin{abstract}
Let $(\xi_1,\eta_1),(\xi_2,\eta_2),\ldots$ be a sequence of
i.i.d.\;copies of a random vector $(\xi,\eta)$ taking values in
$\R^2$, and let $S_n := \xi_1+\ldots+\xi_n$. The sequence
$(S_{n-1} + \eta_n)_{n \geq 1}$ is then called perturbed random
walk.

We study random quantities defined in terms of the perturbed
random walk: $\tau(x)$, the first time the perturbed random walk
exits the interval $(-\infty,x]$, $N(x)$, the number of visits to
the interval $(-\infty,x]$, and $\rho(x)$, the last time the
perturbed random walk visits the interval $(-\infty,x]$. We provide
criteria for the a.s.\ finiteness and for the finiteness
of exponential moments of these quantities. Further, we provide
criteria for the finiteness of power moments of $N(x)$ and $\rho(x)$.

In the course of the proofs of our main results, we investigate
the finiteness of power and exponential moments of shot-noise processes
and provide complete criteria for both, power and exponential moments.

\hphantom{\emph{2010 Mathematics Subject Classification:}}		Secondary:  60G40			

\end{abstract}

\section{Introduction}  \label{sec:intro}

The purpose of this article is to study the moments of certain basic renewal-theoretic quantities for a class of perturbed random walks formally defined below. Such random sequences arise as derived processes in various areas of applied probability and we refer to Sect.\;\ref{subsec:applications} for a number of examples.
It is an interesting question and in fact the main motivation behind this work to what extent classical moment results for ordinary random walks must be adjusted in the presence of a perturbating sequence.

\subsection{Setup}	\label{subsec:setup}

Let $(\xi_1, \eta_1), (\xi_2, \eta_2), \ldots$ be a sequence of i.i.d.\;two-dimensional random vectors with generic copy $(\xi,\eta)$.
For notational convenience, we assume that $(\xi,\eta)$ is defined on the same probability space as the $(\xi_k,\eta_k)$, $k \geq 1$ and independent of this sequence.
No condition is imposed on the dependence structure between $\xi$ and $\eta$. Let $(S_n)_{n \geq 0}$ be the zero-delayed ordinary random walk with increments $\xi_{n}$ for $n\in\N$, \textit{i.e.}, $S_0 = 0$ and $S_n = \xi_1+\ldots+\xi_n$, $n \in \N$. Then define its perturbed variant $(T_{n})_{n \geq 1}$, called \emph{perturbed random walk (PRW)}, by
\begin{equation}    \label{eq:PRW}
T_n := S_{n-1} + \eta_n,	\quad	n \in \N.
\end{equation}
It has appeared in a number of recent publications, see for instance \cite{AraGlynn:06, GneIksMar:10b, Iks:07, PalZwart:10}. Here we should mention that, motivated by certain problems in sequential statistics, a very different class of perturbations, which may roughly be characterized by having slowly varying paths in a stochastic sense, has been considered under the label ``nonlinear renewal theory'', see \cite{LaiSieg:77,LaiSieg:79,Woodroofe:82} and \cite[Section 6]{Gut:09}.

For $x\in\R$, define the level $x$ first passage time
\begin{equation}    \label{eq:tau}
\tau(x) := \inf\{n\in\N: T_n>x\},
\end{equation}
the number of visits to $(-\infty, x]$
\begin{equation}    \label{eq:N}
N(x) := \#\{n\in\N: T_n\leq x\},
\end{equation}
and the associated last exit time
\begin{equation}\label{eq:rho}
\rho(x) ~:=~ \sup\{n \in \N: T_n \leq x\}
\end{equation}
with the usual conventions that $\sup \emptyset := 0$ and $\inf \emptyset := \infty$.Our aim is to find criteria for the a.s.\ finiteness of these quantities and for the finiteness of their power and exponential moments.

Let us further denote by $\tau^*(x)$, $N^*(x)$ and $\rho^*(x)$ the corresponding quantities for the ordinary random walk $(S_{n})_{n \geq 0}$ which is obtained in the special case $\eta=0$ a.s.\ after a time shift. If $\xi=0$, then $(T_{n})_{n \geq 1}$ reduces to a sequence of i.i.d.\;r.v.'s.
In this case, $N(x)=\rho(x)=\infty$ a.s.\ and $\tau(x)$ has a geometric distribution whenever $0 < \Prob\{\eta\leq x\} < 1$.
Neither of the two afore-mentioned cases will be subject of our analysis and therefore be ruled out by making the
\noindent
\begin{center}
\textbf{Standing Assumption}:	\quad	$\Prob\{\xi=0\}<1$ and $\Prob\{\eta=0\}<1$.
\end{center}

\subsection{Examples and Applications}	\label{subsec:applications}

Functionals of PRW's appear in several areas of applied probability as demonstrated by the following examples.

\begin{Exa}[{\it Perpetuities}]
Provided that $\sum_{n \geq 1} e^{T_n}$ is a.s.\;convergent, this sum is called \emph{perpetuity} due to its interpretation as a sum of discounted payment streams in insurance and finance. Perpetuities have received an enormous amount of attention which by now has led to a more or less complete theory. A partial survey of the relevant literature may be found in \cite{AlsIksRoe:09}, for more recent contributions see \cite{FillHuber:10, Hitc:10, HitcWes:09, HitcWes:11}. Presumably one of the most challenging open problems in the area is to provide sufficient (and close to necessary) conditions for the absolute continuity of the law of a perpetuity. In the light of serious complications that already arise in the ``simple'' case
$\xi={\rm const}<0$ (see \cite{AlsIksRoe:09} for more information), there is only little hope for the issue being settled in the near future.
\end{Exa}

\begin{Exa}[{\it The Bernoulli sieve}]
The \emph{Bernoulli sieve} is an infinite occupancy scheme in a random environment $(P_k)_{k\ge 1}$, where
\begin{equation}\label{17}
P_k := W_1W_2\cdots W_{k-1}(1-W_k), \quad k\in\N,
\end{equation}
and $(W_k)_{k\ge 1}$ are independent copies of a random variable $W$ taking values in $(0,1)$. One may think of balls that, given $(P_k)_{k\ge 1}$, are independently placed into one of infinitely many boxes $1,2,3,\ldots$, the probability for picking box $k$ being $P_{k}$.  Assuming that the number of balls equals $n$, denote by $K_n$ the number of nonempty boxes. If the law of $|\log W|$ is non-lattice, it was shown in \cite{GneIksMar:10b} that the weak convergence of $K_n$, properly centered and normalized, is completely determined by the weak convergence of
\begin{eqnarray*}
N(x)    &:=&    \#\{k\in\N: P_k \geq e^{-x}\}\\
&=&\#\{k\in\N: W_1\cdots W_{k-1}(1-W_k) \geq e^{-x}\},
\quad  x>0,
\end{eqnarray*}
again properly centered and normalized. Notice that $N(x)$ is the number of visits to $(-\infty, x]$ by the PRW generated by the couples $(|\log W_1|,|\log(1-W_1)|), (|\log W_2|, |\log(1-W_2)|), \ldots$. A summary of known results including relevant literature for the Bernoulli sieve can be found in the recent survey \cite{GneIksMar:10a}.
\end{Exa}

\begin{Exa}[{\it Regenerative processes}]
Let $(W(t))_{t \geq 0}$ be a c\`{a}dl\`{a}g process starting at $W(0) = 0$ and drifting to
$-\infty$ a.s. Suppose there exists a zero-delayed renewal sequence of random epochs $(\tau_n)_{n \geq 0}$ such that the segments (also called cycles)
\begin{equation*}
(W(t))_{0 \leq t \leq \tau_1},\ (W(\tau_1+ t) -W(\tau_1))_{0 \leq t \leq \tau_2-\tau_1}, \ldots
\end{equation*}
are i.i.d. Then $(W(t))_{t \geq 0}$ is a \emph{(strong-sense) regenerative process}, see \cite{Asmussen:03}. For $n\in\N$, put
\begin{equation*}
\xi_n := W(\tau_n)-W(\tau_{n-1})	\quad	\text{and}	\quad
\eta_n := \sup_{\tau_{n-1}\leq t<\tau_n} W(t)-W(\tau_{n-1}).
\end{equation*}
Then $(\xi_1,\eta_1), (\xi_2,\eta_2), \ldots$ are i.i.d., and
\begin{equation*}
\sup_{t \geq 0} W(t) ~=~ \sup_{n \geq1}
(\xi_1+\ldots+\xi_{n-1}+\eta_n),
\end{equation*}
\textit{i.e.}, the supremum of the regenerative process can be represented as the supremum of an appropriate PRW. The supremum $M$, say, of a PRW is a relatively simple functional that has received considerable attention in the literature. For instance, the tail behavior of $M$ was
investigated in \cite{AraGlynn:06, Goldie:91, HaoTangWei:09, Iks:07, PalZwart:07, PalZwart:10}. Some moment results on $M$ can be found in \cite{AlsIks:09, AlsIksRoe:09}.
\end{Exa}

\begin{Exa}[{\it Queues and branching processes}]
Suppose that $\xi$ and $\eta$ are both nonnegative and define, for $t \geq 0$,
\begin{equation*}
R(t) ~:=~ \sum_{k=0}^\infty \1_{\{S_k\leq t<S_k+\eta_{k+1}\}}
= \tau^*(t)-N(t),
\quad t \geq 0.
\end{equation*}
In a ${\rm GI}/{\rm G}/\infty$-queuing system, where customers
arrive at times $S_{0} = 0 < S_{1} < S_{2} < \ldots$ and are immediately served
by one of infinitely many idle servers, the service time of the $k$th customer being $\eta_{k+1}$,
$R(t)$ gives the
number of busy servers at time $t \geq 0$. Another interpretation
of $R(t)$ emerges in the context of a degenerate pure immigration
Bellman--Harris branching process in which each individual is
sterile, immigration occurs at the epochs $S_1$, $S_2$ etc., and
the lifetimes of the ancestor and the subsequent immigrants are
$\eta_{1},\eta_{2},\ldots$. Then $R(t)$ gives the number of
particles alive at time $t \geq 0$. The process $(R(t))_{t \geq 0}$ was also
used to model the number of active sessions in a computer network 
\cite{Konstantopoulos+Lin:98,Mikosch+Resnick:06}.
\end{Exa}

\section{Main Results}  \label{sec:main_results}

\subsection{Almost Sure Finiteness}

It is well-known that a nontrivial zero-delayed random walk $(S_{n})_{n \geq 0}$ (\textit{i.e.}\ a random walk starting at the origin with increment distribution not degenerate at $0$) exhibits one of the following three regimes:
\begin{enumerate}
\item[1)] drift to $+\infty$ (positive divergence): $\lim_{n \to \infty}S_{n}=\infty$ a.s.;
\item[2)] drift to $-\infty$ (negative divergence): $\lim_{n \to \infty}S_{n}=-\infty$ a.s.;
\item[3)] oscillation: $\liminf_{n \to \infty}S_{n}=-\infty$ and $\limsup_{n \to \infty}S_{n}=\infty$ a.s.
\end{enumerate}
PRW's exhibit the same trichotomy. In order to state the result precisely some further notation is needed. As usual, let $\xi^+ = \max(\xi,0)$ and $\xi^- = \max(-\xi,0)$. Then, for $x>0$, define
\begin{equation*}    \label{eq:A(x)}
A_\pm(x) := \int_0^x \Prob\{\pm \xi>y\} \, {\rm d}y = \E \min (\xi^\pm, x)
\quad	\text{and}	\quad	J_\pm(x):= \frac{x}{A_\pm(x)},
\end{equation*}
whenever the denominators are nonzero. Notice that $J_{\pm}(x)$ for $x>0$ is well-defined iff $\Prob\{\pm \xi > 0\} > 0$.
In this case, we define $J_{\pm}(0) := \Prob\{\pm \xi > 0\}^{-1}$.
The following theorem, though not stated explicitly in \cite{GolMal:00}, can be read off from the results obtained there.

\begin{Theorem} \label{Thm:global}
Any PRW $(T_{n})_{n \geq 1}$ satisfying the standing assumption is either positively divergent, negatively divergent or oscillating.
Positive divergence takes place iff
\begin{equation}    \label{eq:drift_+infty}
\lim_{n \to \infty} S_n=\infty \quad \text{and} \quad
\E J_+(\eta^-)<\infty,
\end{equation}
while negative divergence takes place iff
\begin{equation}    \label{eq:drift_-infty}
\lim_{n \to \infty} S_n=-\infty \text{ a.s.} \quad \text{and} \quad
\E J_-(\eta^+)<\infty.
\end{equation}
Oscillation occurs in the remaining cases, thus iff either
\begin{equation}\label{eq:oscillation1}
-\infty = \liminf_{n \to \infty} S_n < \limsup_{n \to \infty} S_n = \infty \text{ a.s.,}
\end{equation}
or
\begin{equation}\label{eq:oscillation2}
\lim_{n \to \infty} S_n=\infty \text{ a.s.} \quad \text{and} \quad
\E J_+(\eta^-)=\infty,
\end{equation}
or
\begin{equation}\label{eq:oscillation3}
\lim_{n \to \infty} S_n=-\infty \text{ a.s.}    \quad   \text{and} \quad
\E J_-(\eta^+)=\infty.
\end{equation}
\end{Theorem}

\begin{Rem}	\label{Rem:trichotomy}
As a consequence of Theorem\;\ref{Thm:global} it should be observed that a PRW $(T_n)_{n \geq 1}$ may oscillate even if the corresponding ordinary random walk $(S_n)_{n \geq 0}$ drifts to $\pm \infty$.
\end{Rem}

In view of the previous result it is natural to take a look at the a.s.\ finiteness of the first passage times $\tau(x)$. Plainly, if $\limsup_{n\to\infty}T_{n}=\infty$ a.s., then $\tau(x)<\infty$ a.s.\ for all $x\in\R$. On the other hand, one might expect in the opposite case, \textit{viz.} $\lim_{n\to\infty}T_{n}=-\infty$ a.s., that $\Prob\{\tau(x)=\infty\}>0$ for all $x\ge 0$, for this holds true for ordinary random walks.
Namely, if $\lim_{n\to\infty}S_{n}=-\infty$ a.s., then $\Prob\{\sup_{n \geq 1}S_{n}\leq 0\}=\Prob\{\tau^*=\infty\}>0$.
The following result shows that this conclusion may fail for a PRW. It further provides a criterion for the a.s.\ finiteness of $\tau(x)$ formulated in terms of $(\xi,\eta)$.

\begin{Theorem} \label{Thm:finiteness_tau}
Let $(T_{n})_{n \geq 1}$ be negatively divergent and $x\in\R$. Then $\tau(x)<\infty$ a.s.\ iff $\Prob\{\xi<0, \eta \leq x\}=0$. Furthermore, $\Prob\{\eta \leq x\} < 1$ holds true in this case.
\end{Theorem}

In order to establish a criterion for the a.s.\ finiteness of the r.v.'s $N(x)$ and $\rho(x)$, it only takes to observe that, if one of those is a.s.\ finite for some $x$,
then $\liminf_{n \to \infty}T_{n}>-\infty$ a.s. Hence, by Theorem\;\ref{Thm:global}, $(T_{n})_{n \geq 1}$ must be positively divergent. Since the converse holds trivially true, we can state the following result analogous to the case of ordinary random walks.

\begin{Theorem} \label{Thm:finiteness_rho}
The following assertions are equivalent:
\begin{itemize}
\item[(i)] $(T_{n})_{n \geq 1}$ is positively divergent.
\item[(ii)] $N(x)<\infty$ a.s.\ for some/all $x\in\R$.
\item[(iii)] $\rho(x)<\infty$ a.s.\ for some/all $x\in\R$.
\end{itemize}
\end{Theorem}

\subsection{Finiteness of Exponential Moments}	\label{subsec:finiteness_exponential}

The following theorems are on finiteness of exponential moments of $\tau(x)$, $N(x)$ and $\rho(x)$.

\begin{Theorem} \label{Thm:exponential_tau}
Let $a>0$ and $x\in\R$.

\vspace{.2cm}\noindent
(a) If $\Prob\{\xi < 0, \eta \leq x\}=0$, then $\E \exp(a\tau(x))<\infty$ iff
\begin{equation}    \label{eq:xi=0_eta_leq_x}
e^{a}\,\Prob\{\xi=0, \eta\leq x\}<1.
\end{equation}
(b) If $\Prob\{\xi < 0, \eta \leq x\}>0$, then
\begin{align}
&	\E \exp(a\tau(x))<\infty,\label{eq:Ee^atau<infty}	\\
&	\E \exp(a\tau(y))<\infty\quad\text{ for all }y\in\R,\label{eq:Ee^atau(y)<infty}	\\
&	\E \exp(a\tau^*)<\infty,\label{eq:E^atau*<infty}	\\
&	R := -\log \inf_{t \geq 0}\,\E e^{-t\xi}\ge a	\label{eq:-loginf E^-txi>=a}
\end{align}
are equivalent assertions.
\end{Theorem}

Turning to exponential moments of $N(x)$, the number of visits of
$(T_{n})_{n \geq 1}$ to $(-\infty,x]$, for $x \in \R$, let us point
out before-hand that these random variables are a.s.\ finite iff
$(T_{n})_{n\ge 1}$ is positively divergent which in turn holds
true iff $(S_{n})_{n \geq 0}$ is positively divergent and
\begin{equation}	\label{eq:EJ+eta-<infty}
\E J_{+}(\eta^{-})<\infty
\end{equation}
(see Theorem\;\ref{Thm:global}) which will therefore be assumed hereafter.

\begin{Theorem} \label{Thm:exponential_N}
Let $(T_{n})_{n \geq 1}$ be a positively divergent PRW.

\vspace{.2cm}\noindent
(a) If $\xi \geq 0$ a.s., then the assertions
\begin{align}
&\E \exp(a N(x))<\infty,	\label{eq:Ee^aN<infty}\\
&e^a\,\Prob\{\xi=0,\eta\leq x\}+\Prob\{\xi=0,\eta>x\}<1	\label{eq:e^aP(xi=0,eta<=x)+P(xi=0,eta>x)<1}
\end{align}
are equivalent for each $a>0$ and $x\in\R$. As a consequence,
\begin{equation}\label{eq:mgf N(x)}
\left\{a>0:\E e^{aN(x)}<\infty\right\}=(0,a(x))
\end{equation}
for any $x\in\R$, where $a(x)\in (0,\infty]$ equals the supremum of all positive $a$ satisfying \eqref{eq:e^aP(xi=0,eta<=x)+P(xi=0,eta>x)<1}. As a function of $x$, $a(x)$ is nonincreasing with lower bound $-\log\Prob\{\xi=0\}$.

\vspace{.2cm}\noindent
(b) If $\xi>0$ a.s., then $a(x)=\infty$ for all $x\in\R$, thus $\E e^{aN(x)}<\infty$ for any $a>0$ and $x\in\R$.

\vspace{.2cm}\noindent (c) If $\Prob\{\xi < 0\} > 0$, then the
following assertions are equivalent:
\begin{align}
&   \E \exp(a N(x)) < \infty   \quad   \text{for some/all } x \in \R,   \label{eq:Ee^aN(x)<infty}   \\
&   \E \exp(a N^*(x)) < \infty     \quad   \text{for some/all } x \in \R,   \label{eq:Ee^aN*(x)<infty}  \\
&   R=-\log \inf_{t \geq 0}\,\E e^{-t\xi} \geq a   \label{eq:R_geq_a}.
\end{align}
\end{Theorem}

\begin{Theorem} \label{Thm:exponential_rho}
Let $(T_{n})_{n \geq 1}$ be a positively divergent PRW, $a>0$ and $R = -\log \inf_{t \geq 0} \E e^{-t\xi}$.
\begin{itemize}
    \item[(a)]
    Assume that $\Prob\{\xi \geq 0\} = 1$. Let $x\in\R$ and assume that $\Prob\{\eta \leq x\} > 0$.
    Then the following assertions are equivalent:
    \begin{equation}    \label{eq:Ee^arho<infty}
    \E \exp(a\rho(x)) < \infty;
    \end{equation}
    \begin{equation}    \label{eq:V_a(y)<infty}
    V_a(y)    ~:=~ \sum_{n \geq 1} e^{an} \Prob\{T_n \leq y\} < \infty \ \text{ for some/all } y \geq x;
    \end{equation}
    \begin{equation}    \label{eq:a<-logP(xi=0)_Ee^-gammaeta<infty}
    a < -\log \Prob\{\xi=0\} \text{ and } \E e^{-\gamma \eta} < \infty,
    \end{equation}
    where $\gamma$ is the unique positive number satisfying $\E e^{-\gamma \xi}=e^{-a}$.
    \item[(b)]
    If $\Prob\{\xi < 0\} > 0$, then the following assertions are equivalent:
    \begin{equation}    \label{eq:Ee^arho(x)<infty_fax}
    \E \exp(a\rho(x)) < \infty \quad   \text{for some/all } x\in\R;
    \end{equation}
    \begin{equation}    \label{eq:V_a(x)<infty_fax}
    V_a(x) ~=~ \sum_{n \geq 1} e^{an} \Prob\{T_n \leq x\} < \infty \ \text{ for some/all } x \in\R;
    \end{equation}
    \begin{equation}    \label{eq:a<=R_extra}
    a < R \text{ and } \E e^{-\gamma\eta}<\infty
    \;\text{ or }\;
    a=R,\ \E \xi e^{-\gamma \xi} < 0 \text{ and }  \E e^{-\gamma\eta}<\infty
    \end{equation}
    where $\gamma$ is the minimal positive number satisfying $\E e^{-\gamma \xi}=e^{-a}$.
\end{itemize}
\end{Theorem}

\begin{Rem}	\label{Rem:exponential_rho}
Notice that in Theorem\;\ref{Thm:exponential_rho} the case $\Prob\{\xi \geq 0\}=1$, $\Prob\{\eta \leq x\} = 0$ is not treated.
But this case is trivial since then $\rho(x) = 0$ a.s., \textit{cf.}\ Lemma\;\ref{Lem:rho_trivial}.
\end{Rem}

\subsection{Finiteness of Power Moments}	\label{subsec:finiteness_power}

\begin{Theorem}	\label{Thm:power_N}
Let $(T_n)_{n \geq 0}$ be a positively divergent PRW and $p>0$. The following conditions are equivalent:
\begin{equation}    \label{eq:EN^p<infty}
\E N(x)^p < \infty	\text{ for some/all }  x \in \R;
\end{equation}
\begin{equation}    \label{eq:EN*^p<infty}
\E N^*(x)^p < \infty	\text{ for some/all } x \geq 0;
\end{equation}
\begin{equation}    \label{eq:criterion_EN*^p<infty}
\E J_+(\xi^-)^{p+1} < \infty.
\end{equation}
\end{Theorem}

\begin{Theorem} \label{Thm:power_rho}
Let $(T_n)_{n \geq 0}$ be a positively divergent PRW and $p>0$. Then the following assertions are equivalent:
\begin{equation}    \label{eq:Erho^p<infty}
\E \rho(x)^p<\infty    \text{ for some/all }	x \in \R;
\end{equation}
\begin{equation}    \label{eq:Erho*^p<infty}
\E \rho^*(y)^p<\infty \text{ for some/all } y \geq 0 \quad \text{and} \quad \E J_+(\eta^-)^{p+1} < \infty;
\end{equation}
\begin{equation}    \label{eq:criterion_Erho^p<infty}
\E J_+(\xi^-)^{p+1} < \infty \quad \text{and} \quad \E J_+(\eta^-)^{p+1} < \infty.
\end{equation}
\end{Theorem}

\begin{Rem}
According to Theorem\;\ref{Thm:exponential_rho}, for fixed $a>0$,
\begin{eqnarray*}
\E e^{a\rho(x)} & < & \infty		\quad	\text{for some/all } x \in \R	\quad   \text{iff}	\\
\sum_{n \geq 1} e^{an} \Prob\{T_n\leq x\} & < & \infty \quad	\text{for some/all } x \in \R.
\end{eqnarray*}
According to
\cite[Theorem\;2.1]{Kesten+Maller:96}, for fixed $p>0$,
\begin{eqnarray*}
\E \rho^*(x)^p & < & \infty \quad \text{for some/all } x \geq 0 \quad \text{iff}	\\
\sum_{n \geq 1} n^{p-1} \Prob\{S_n\leq x\} & < & \infty \quad \text{for some/all } x \geq 0.
\end{eqnarray*}
In the light of these results it may appear to be unexpected that,
in general, the finiteness of $\E \rho(x)^p$ is not equivalent to
the convergence of the series $\sum_{n \geq 1}n^{p-1} \Prob\{T_n \leq x\}$.
Indeed, it can be checked (but we omit the details) that a
criterion for the convergence of the latter series is as follows:
\begin{equation*}
\E \rho^*(x)^p ~<~ \infty \quad \text{for some/all } y \geq 0  \quad \text{and} \quad \E J_+(\eta^-)^p<\infty.
\end{equation*}
\end{Rem}

\subsection{Open Problems}	\label{subsec:Open_problems}

The preceding subsections give complete characterizations of the finiteness of exponential moments of $\tau(x)$, $N(x)$ and $\rho(x)$ as well as the finiteness of power moments of $N(x)$ and $\rho(x)$.
In view of this, it is natural to ask for a criterion for the finiteness of power moments of $\tau(x)$.
This question is very delicate.

First of all, it is worth mentioning that in case $p \in (0,1)$, for oscillating random walks,
no criterion for the finiteness of $\E \tau^*(0)$ is known though some partial results have been obtained.
For instance, from Wald's equation it follows immediately that $\E \tau^*(x) = \infty$ when $\E \xi$ exists and equals $0$.
Using an extension of this argument, it has been shown that if $\E |\xi|^p < \infty$ for some $1 < p \leq 2$ and $\E \xi = 0$,
then $\E (\tau^*)^{1/p} = \infty$
(see \cite{Burkholder+Gundy:70} for the case $p=2$ and \cite{Chow+Robbins:91} for the case $1<p<2$).
Further, it was shown in \cite{Chow:94} that if $\xi$ is concentrated on $\{-1,0,1,2,3,\ldots\}$, then $\E (\tau^*)^{1-1/p} < \infty$ iff
\begin{equation}	\label{eq:Chows_criterion}
\int_0^{\infty} \Prob\{\xi > t\}^{1/p} \, {\mathrm d}t ~<~\infty.
\end{equation}
Recently, progress in this matter has been achieved in \cite{Uchiyama:11}.

A criterion for $\E \tau(x)^p < \infty$ in the oscillating case would presumably be connected to a criterion for the finiteness of $\E \tau^*(x)^p$.

Even when the underlying $(S_n)_{n \geq 0}$ is positively divergent, a criterion for $\E \tau^p < \infty$ is not easy to obtain.
It is not hard to see that $\E \tau^*(x)^p < \infty$ for some $x \geq 0$ implies $\E \tau(x)^p < \infty$ for all $x \geq 0$:

\begin{Prop}	\label{Prop:Etau*^p<infty=>Etau^p<infty}
Let $p>0$. When $\E (\tau^*)^p < \infty$, then $\E \tau(x)^p < \infty$ for all $x \geq 0$.
\end{Prop}

Unlike in the situation of exponential moments
and caused by possible big jumps to the right coming from the perturbations $\eta_n$,
the converse implication does not hold.
In other words, $\E (\tau^*(0))^p < \infty$ is not necessary for $\E \tau(x)^p < \infty$ to hold.
This is observed in the following proposition:

\begin{Prop}	\label{Prop:Etau*=infty_and_Etau<infty}
Fix $p > 0$. Assume that for some $c \geq 0$, $\Prob\{\xi \geq -c\} = 1$
and set $\sigma(x)	 :=	\inf\{k \in \N: \eta_k-(k-1)c > x\}$.
Then $\E \sigma(x)^p<\infty$ implies $\E \tau(x)^p<\infty$.

Moreover, if $\lim_{y \to \infty} y\Prob\{\eta>y\} = s \in [0,\infty]$,
then $\E \sigma(x)^p$ is finite or infinite according to whether $s>cp$ or $s<cp$.
\end{Prop}

From this proposition it immediately follows that $\E \tau(x)^p < \infty$ does not imply $\E (\tau^*)^p < \infty$ .
Indeed, when $\xi \geq -1$ a.s.\ but $\xi$ is such that $(S_n)_{n \geq 0}$ is drifting to $-\infty$ or oscillating,
then $\E \tau^*(0) = \infty$ and, hence, $\E \tau^*(0)^p = \infty$ for all $p \geq 1$.
On the other hand, when $\eta$ is such that $\lim_{y \to \infty} y\Prob\{\eta>y\} = \infty$,
then $\E \sigma(0)^p < \infty$ for all $p > 0$ and, consequently, $\E \tau(0)^p < \infty$ for all $p > 0$.

\subsection{Notation and Overview}	\label{subsec:notation}

At this point, we introduce some notation which is used throughout the article.
First of all, whenever it is convenient, we write $\tau$, $N$ and $\rho$ for $\tau(0)$, $N(0)$ and $\rho(0)$, respectively.
Analogously, we write $\tau^*$, $N^*$ and $\rho^*$ for $\tau^*(0)$, $N^*(0)$ and $\rho^*(0)$, respectively.

As usual, $f(t) \sim g(t)$ as $t \to \infty$ for functions $f$ and $g$, means that $f(t)/g(t) \to 1$ as $t \to \infty$.
Similarly, $f(t) \asymp g(t)$ as $t \to \infty$ means that $0 < \liminf_{t \to \infty }f(t)/g(t) \leq \limsup_{t \to \infty} f(t)/g(t) < \infty$.

We finish this section with an overview over the further organization of the article.
The proofs of the main results are given in Sect.\;\ref{sec:Proofs}.
The proofs concerning finiteness of moments of $N(x)$, Theorems\;\ref{Thm:exponential_N} and \ref{Thm:power_N},
are based on general results on finiteness of (exponential and power) moments of shot-noise processes.
These results and their proofs can be found in Sect.\;\ref{sec:shot-noise}.
The appendix contains auxiliary results from random walk theory
(Sect.\;\ref{subsec:aux_RW}) and some elementary facts (Sect.\;\ref{subsec:aux_elementary}).

\section{Shot-Noise Processes} \label{sec:shot-noise}

Let $\xi$ be a real-valued random variable with $\Prob\{\xi=0\}<1$ and $(X(t))_{t \in \R}$ a doubly infinite nonnegative stochastic process with nondecreasing paths such that $\lim_{t\to-\infty}X(t)=0$ a.s. Any dependence between $(X(t))_{t\in\R}$ and $\xi$ is allowed. Further, given a sequence $((X_n(t))_{t \in \R},\xi_n))_{n \geq 1}$ of independent copies of $((X(t))_{t \in \R}, \xi)$, define
\begin{equation*}
S_0:=0, \quad   S_n:=\xi_1+\cdots+\xi_n,    \quad   n \in \N,
\end{equation*}
and then the \emph{renewal shot-noise process} $Z(\cdot)$ with random response functions
$X_n(\cdot)$ by
\begin{equation*}
Z(t) ~:=~ \sum_{n \geq 1} X_n(t-S_{n-1}), \quad   t \in \R.
\end{equation*}

\subsection{Examples of Shot-Noise Processes}	\label{subsec:shot-noise_examples}

In this section, we give some examples of shot-noise processes.

\begin{Exa}  \label{Exa:current}
The current at time $t$ induced by an electron that arrives at time $s$ at the anode of a vacuum tube equals $f(t-s)$ for some appropriate deterministic response function $f$ vanishing on the negative halfline. Assuming that $X(t)=f(t)$ and the $S_n$ are the arrival times in a homogeneous Poisson process, the total current at time $t$ equals
\begin{equation*}
Z(t) ~=~ \sum_{n \geq 1} f(t-S_{n-1}),    \quad   t \geq 0.
\end{equation*}
This is the classical shot-noise process \cite{Schottky:18}.
\end{Exa}

\begin{Exa}  \label{Exa:shot-noise_deterministic_response}
A very popular model in the literature has $X(\cdot)$ in multiplicative form $X(t)=\eta f(t)$ for a nonnegative random variable $\eta$ and some deterministic $f$ (see \cite{DoneyOBrian:91, Iks:01, Iks:12, Lebedev:02, McCormick:97, Rice:77, Takacs:56a} and the references therein). In the particular case $f(t)=e^{at}$ for some $a \not = 0$, the corresponding shot-noise process is a perpetuity, namely
\begin{equation*}
Z(t)=e^{at}\sum_{n \geq 1}e^{-aS_{n-1}} \eta_n, \quad t \in \R.
\end{equation*}
\end{Exa}

The moment results for shot-noise processes we are going to derive hereafter will be a key  in the analysis of the moments of $N(t)$, the number of visits to $(-\infty,t]$ of a PRW $(T_n)_{n\ge 1}$. The link between $N(t)$ and shot-noise processes is disclosed in the following example.

\begin{Exa}  \label{Exa:Z(t)=N(t)}
If $X_n(t)=\1_{\{\eta_n \leq t\}}$ for a real-valued random variable $\eta_n$, $n \geq 1$, then $Z(t)$ equals the number of visits to $(-\infty, t]$ of the PRW $(S_{n-1}+\eta_n)_{n \geq 1}$, thus $Z(t)=N(t)$.
\end{Exa}

\subsection{Finiteness of Exponential Moments of Shot-Noise \\ Processes}	\label{subsec:finitieness_exponential_shot-noise}

Our first moment result for shot-noise processes, assuming $\xi \geq 0$ a.s., provides two conditions which combined are necessary and sufficient for the finiteness of $\E e^{aZ(t)}$ for fixed $a>0$ and $t\in\R$. As before, let $\tau^*(x)=\inf\{n\ge 1:S_{n}>x\}$. Moreover, we denote by $\mathbb{U}:=\sum_{n \geq 0} \Prob\{S_n \in \cdot\}$ the renewal measure associated with $(S_{n})_{n \geq 0}$.

\begin{Theorem}   \label{Thm:shot-noise_positive}
Let $\xi\ge 0$ a.s. Then, for any $a>0$ and $t\in\R$,
\begin{align}\label{eq:Ee^aZ(t)<infty}
\E e^{aZ(t)}<\infty
\end{align}
holds if and only if
\begin{align}
&r(t):=\int \bigg(\E e^{aX(t-y)}-1\bigg)\,\mathbb{U}({\rm d}y)<\infty	\label{eq:r(t)<infty}\\
\text{and}\quad&l(t):=\E\left(\prod_{n=1}^{\tau^*}e^{aX_{n}(t-S_{n-1})}\right)<\infty.	\label{eq:l(t)<infty}
\end{align}
Moreover, \eqref{eq:r(t)<infty} alone implies $\E e^{aZ(t_{0})}<\infty$ for some $t_{0}\leq t$.
\end{Theorem}

\begin{Rem}	\label{Rem:shot-noise_positive}
It can be extracted from the proof given next that we may replace $\tau^*$ in \eqref{eq:l(t)<infty} by any other $(S_{n})_{n \geq 0}$-stopping time $\tau\ge\tau^*$. Note also that, unlike the case when $\Prob\{\xi<0\}>0$ to be discussed later, $\tau^*$ coincides with $\widehat{\tau}:=\inf\{n\ge 1:\xi_{n}>0\}$ and thus has a geometric distribution with parameter $\Prob\{\xi>0\}$. Finally, \eqref{eq:l(t)<infty} is a trivial consequence of \eqref{eq:r(t)<infty} if $\xi>0$ a.s.
\end{Rem}

\begin{proof}
Observe that
\begin{align}
\begin{split}\label{eq:lower bound}
e^{aZ(t)}-1 ~&=~ \sum_{n\ge 1}\bigg(e^{aX_n(t-S_{n-1})}-1\bigg)\prod_{k \geq n+1}
e^{aX_{k}(t-S_{k-1})}\\
&\geq\ \sum_{n\ge 1}\bigg(e^{aX_n(t-S_{n-1})}-1\bigg)
\end{split}\\
\text{and}\quad e^{aZ(t)} ~&\geq~ \prod_{n=1}^{\tau^*}e^{aX_{n}(t-S_{n-1})}
\label{l(t) lower bound}
\end{align}
hold whenever $Z(t) < \infty$. Taking expectations in the above inequalities therefore gives the implications
``\eqref{eq:Ee^aZ(t)<infty}$\Rightarrow$\eqref{eq:r(t)<infty}'' and ``\eqref{eq:Ee^aZ(t)<infty}$\Rightarrow$\eqref{eq:l(t)<infty}''.

\vspace{.2cm}
In turn, assume that \eqref{eq:r(t)<infty} and \eqref{eq:l(t)<infty} hold.
Let $(\tau_n^*)_{n \geq 0}$ be the zero-delayed renewal sequence of strictly ascending ladder epochs of $(S_{n})_{n \geq 0}$, thus $\tau_{1}^*=\tau^*$ and define
\begin{equation*}
L(s) ~:=~ \prod_{n=1}^{\tau^*}e^{aX_{n}(s-S_{n-1})}
\end{equation*}
for $s \in \R$. Then $\E L(s) \leq \E L(t)<\infty$ for all $s \leq t$, for $L(\cdot)$ is nondecreasing and \eqref{eq:l(t)<infty} holds.
Pick $\varepsilon>0$ so small that
\begin{equation*}
\E L(s)\1_{\{S_{\tau^*}\leq \varepsilon\}} \leq \beta := \E L(t)\1_{\{S_{\tau^*} \leq \varepsilon\}} < 1
\end{equation*}
for all $s \leq t$. Next define
$Z_{0}(\cdot)=Z_{0}'(\cdot)=0$ and
\begin{equation*}
Z_{n}(\cdot) ~:=~ \sum_{k=1}^{n} X_{k}(\cdot-S_{k-1}),	\quad
Z_{n}'(\cdot) ~:=~ \sum_{k=\tau^*+1}^{\tau^*+n} X_{k}(\cdot-(S_{k-1}-S_{\tau^*}))
\end{equation*}
for $n\in\N$. Plainly, $Z_{n}(\cdot)\uparrow Z(\cdot)$ and similarly
\begin{equation*}
Z_{n}'(\cdot)\uparrow Z'(\cdot) ~:=~ \sum_{n \geq \tau^*+1}X_{n}(\cdot-(S_{n-1}-S_{\tau^*}))
\end{equation*}
as $n\to\infty$. Note that, each $Z_{n}'(\cdot)$ is a copy of $Z_{n}(\cdot)$ and further independent of $(L(\cdot),S_{\tau^*})$.
Now observe that
\begin{align*}
Z_{n}(t) ~\le~ &
Z_{\tau^*}(t)+Z_{n-\tau^*}'(t)\1_{\{\tau^*\leq n,S_{\tau^*}\leq \varepsilon\}}	\\
& +Z_{n-\tau^*}'(t-\varepsilon)\1_{\{\tau^*\leq n,S_{\tau^*}>\varepsilon\}}		\\
~\le~& Z_{\tau^*}(t)+Z_{n}'(t)\1_{\{S_{\tau^*}\leq \varepsilon\}}
+Z_{n}'(t-\varepsilon)\1_{\{S_{\tau^*}>\varepsilon\}}
\end{align*}
and therefore, using the stated independence properties,
\begin{align}
\E e^{aZ_{n}(t)} ~&\leq~
\E\left(L(t)\1_{\{S_{\tau^*}\leq \varepsilon\}}e^{aZ_n'(t)}+L(t)\1_{\{S_{\tau^*}>\varepsilon\}}e^{aZ_n'(t-\varepsilon)}\right)	\notag	\\
&\le~	\beta\,\E e^{aZ_{n}(t)}+\E L(t)\,\E e^{aZ_{n}(t-\varepsilon)}		\label{eq:Z_n(t)}
\end{align}
for any $n \in \N$.
Now notice that by \eqref{eq:r(t)<infty}, $\E e^{aX_1(t)} < \infty$ and hence, by Fubini's theorem,
\begin{equation}	\label{eq:Ee^aZ_n(t)<infty}
\E e^{aZ_n(t)} ~\leq~ \E \prod_{k=1}^n e^{aX_k(t)} ~=~ \left(\E e^{a X_1(t)} \right)^n ~<~ \infty.
\end{equation}
By solving \eqref{eq:Z_n(t)} for $\E e^{aZ_{n}(t)}$ and letting $n\to\infty$, we arrive at
\begin{equation*}
\E e^{aZ(t)} ~\leq~ (1-\beta)^{-1}\,\E L(t)\,\E e^{aZ(t-\varepsilon)}
\end{equation*}
and then upon successively repeating this argument at
\begin{equation*}
\E e^{aZ(t)} ~\leq~ (1-\beta)^{-n} \E e^{aZ(t-n\varepsilon)} \,\prod_{k=0}^{n-1} \E L(t-k \varepsilon)
\end{equation*}
for any $n\in\N$.
Hence, $\E e^{aZ(t)}<\infty$ as claimed if we verify $\E e^{aZ(t_{0})}<\infty$ for some $t_{0}<t$.

To this end, pick $t_{0}$ such that $r(t_{0})<1$ which is possible
because \eqref{eq:r(t)<infty} in combination with the monotone convergence theorem entails $\lim_{t\to -\infty}r(t)=0$.
Note also that $r(t_{0})<\infty$ implies $\E e^{aX(t_{0})}<\infty$. Define
\begin{align*}
b_{n}	&	:=	\E e^{aZ_{n}(t_0)}
\quad	\text{and}	\quad	c_{n}:=\sum_{k=1}^{n} \E \bigg(e^{aX_{k}(t_0-S_{k-1})}-1\bigg)
\end{align*}
for $n \in \N_{0}$, in particular, $b_{0}=1$, $c_0=0$. The $b_n$'s are finite by the same argument as in \eqref{eq:Ee^aZ_n(t)<infty}. Moreover, $\sup_{n\ge 1}c_n=r(t_{0})<1$. With this notation and for any $n\in\N$, we obtain (under the usual convention that empty products are defined as 1)
\begin{align*}
e^{aZ_{n}(t_0)}-1 ~&=~
\sum_{k=1}^{n}\bigg(e^{aX_{k}(t_0-S_{k-1})}-1\bigg)\prod_{j=k+1}^{n}e^{aX_{j}(t_0-S_{j-1})}	\\
& \le~ \sum_{k=1}^{n}\bigg(e^{aX_{k}(t_0-S_{k-1})}-1\bigg)\prod_{j=k+1}^{n} e^{aX_{j}(t_0-S_{j-1}+S_{k})}	\\
& \le~ \sum_{k=1}^{n}\bigg(e^{aX_{k}(t_0-S_{k-1})}-1\bigg)\prod_{j=k+1}^{k+n-1} e^{aX_{j}(t_0-S_{j-1}+S_{k})}.
\end{align*}
For fixed $k,n \in \N$, the random variable $\prod_{j=k+1}^{k+n-1}e^{aX_j(t-S_{j-1}+S_k)}$ is independent of $e^{aX_k(t_0-S_{k-1})}$ and has the same law as $e^{aZ_{n-1}(t_0)}$. Taking expectations, we get
\begin{equation*}b_{n}-1\leq c_{n}b_{n-1}\leq r(t_{0})b_{n-1}\quad\text{for }n\in\N \end{equation*}
and thereupon at $b_n\leq (1-r(t_{0}))^{-1}$ for all $n\in\N$. Finally letting $n\to\infty$, we conclude $\E e^{aZ(t_0)}<\infty$.

The previous argument has only used \eqref{eq:r(t)<infty} and thus also shown the last assertion of the theorem.
\end{proof}

We will now carry over the previous result to the case when $(S_{n})_{n \geq 0}$ is a positively divergent random walk taking negative values with positive probability.
As before, let $\mathbb{U}$ be the pertinent intensity measure and $\mathbb{U}^{>}$ the renewal measure of the associated renewal process
$(S_{n}^{>})_{n \geq 0}$, say, of strictly ascending ladder heights with increments $\xi_{n}^{>} = S_n^> - S_{n-1}^>$, $n \in \N$.
The corresponding ladder epochs are denoted as $\tau_{n}^*$ for $n \in \N$, thus $\tau_{1}^* = \tau^*$.
Further, defining $M^*$ to be a generic copy of $\inf_{n \geq 0} S_{n}$ that is independent of any other occurring random variable, a well-known identity in the fluctuation theory of random walks (see \textit{e.g.}\ \cite[Theorem\;VIII.2.2]{Asmussen:03} after a change of sign) states that
\begin{equation}\label{eq:renewal measure - minimum}
\mathbb{U}(B)	~=~	(\E \tau^*) \,Q*\mathbb{U}^{>}(B)	~=~	(\E \tau^*) \,\E \mathbb{U}^{>}(B-M^*)
\end{equation}
for all Borel subsets $B$ of $\R$, where $Q := \Prob\{M^* \in \cdot\}$ and $*$ denotes convolution.

\begin{Theorem}   \label{Thm:shot-noise_general}
Let $(S_{n})_{n \geq 0}$ be positively divergent and $\Prob\{\xi<0\}>0$. Then the following assertions are equivalent for any $a>0$:
\begin{align}
\E e^{aZ(t)} < \infty	&	\quad \text{for some } t\in\R,\label{3236}\\
\E e^{aZ(t)} < \infty	&	\quad \text{for all } t\in\R,\label{3237}\\
r^{>}(t) < \infty			&	\quad \text{for all } t\in\R,\label{3238}\\
r^{>}(t) < \infty			&	\quad \text{for some } t\in\R,\label{3239}
\end{align}
where $l(t)$ is defined as in \eqref{eq:l(t)<infty} and
\begin{equation*}
r^{>}(t):=\int (l(t-u)\!-\!1)\,\mathbb{U}^{>}({\mathrm d}u)
\end{equation*}
for $t\in\R$. Furthermore, the conditions imply $r(t)<\infty$ and $l(t)<\infty$ for all $t\in\R$.
\end{Theorem}

\begin{proof}
The last assertion follows from \eqref{eq:lower bound} and $l(t)-1 \leq r^{>}(t)$.

\vspace{.2cm}
``\eqref{3236}$\Rightarrow$\eqref{3237}'' Put $g(t):=\E e^{aZ(t)}$ for $t\in\R$ and use the first line of \eqref{eq:lower bound} to infer via conditioning and with the help of \eqref{eq:renewal measure - minimum}
\begin{align*}
g(t)-1\ &=~ \sum_{n \geq 1}\E\Bigg(\Big(e^{aX_n(t-S_{n-1})}-1\Big)\prod_{k\ge n+1}
e^{aX_{k}(t-S_{k-1})}\Bigg)\\
&=~ \sum_{n\ge 1}\E\Bigg(\Big(e^{aX_n(t-S_{n-1})}-1\Big)\E\Bigg(\prod_{k\ge n+1}
e^{aX_{k}(t-S_{k-1})}\bigg|S_{n}\Bigg)\Bigg)\\
&=~ \sum_{n\ge 1}\E\bigg(\Big(e^{aX_n(t-S_{n-1})}-1\Big)g(t-S_{n})\bigg)\\
&=~ \int_{[0,\infty)}\E\bigg(\Big(e^{aX(t-y)}-1\Big)g(t-y-\xi)\bigg)\ \mathbb{U}({\mathrm d}y)\\
&=~ (\E \tau^*) \int_{[0,\infty)}\E\bigg(\Big(e^{aX(t-y-M^*)}-1\Big)g(t-y-\xi-M^*)\bigg)\ \mathbb{U}^{>}({\mathrm d}y)\\
&\geq~ \E\bigg(\Big(e^{aX(t-M^*)}-1\Big)g(t-\xi-M^*)\bigg)
\end{align*}
for any $t\in\R$. But $\Prob\{\xi<0\}>0$ implies $\Prob\{\xi+M^*<-x\}>0$ for all $x>0$
(notice that $\xi+M^* \stackrel{d}{=} \inf_{n \geq 1}S_{n}$). Consequently, $g(t+x)<\infty$ for any $x>0$ if $g(t)<\infty$.
By monotonicity, we also have $g(t+x) < \infty$ for $x<0$.

\vspace{.2cm}
``\eqref{3237}$\Rightarrow$\eqref{3238}''
Put
\begin{equation*}
L_{n}(s) ~:=~
\prod_{k=\tau_{n-1}^*+1}^{\tau_{n}^*} \exp\left(aX_{k}(s-\big(S_{k-1}-S_{\tau_{n-1}^*})\big)\right)
\end{equation*}
for $n\in\N$ and $s\in\R$ which are i.i.d.\;with $L_{1}(s)=L(s)$ as defined in the proof of Theorem\;\ref{Thm:shot-noise_positive}. If $\E e^{aZ(t)}<\infty$, then
\begin{align}
e^{aZ(t)}-1\ &=~ \sum_{n\ge 1}\left(L_{n}(t-S_{\tau_{n-1}^*})-1\right)\prod_{k\ge n+1}L_{k}(t-S_{\tau_{k-1}^*})	\nonumber	\\
&\geq\ \sum_{n\ge 1}\left(L_{n}(t-S_{\tau_{n-1}^*})-1\right).	\label{eq:lower bound2}
\end{align}
Taking expectations on both sides of this inequality gives $r^{>}(t)<\infty$.

\vspace{.2cm}
``\eqref{3239}$\Rightarrow$\eqref{3236}'' If $r^{>}(t)$ for some $t \in \R$, then also $l(t)<\infty$ and, therefore, $r^{>}(t_{0})<1$ and $l(t_{0})-1<1$ for some $t_{0}\leq t$. Since
\begin{equation*}
e^{aZ_{\tau_{n}^*}(s)} ~\leq~ \prod_{k=1}^{n}L_{k}(s),
\end{equation*}
we infer
\begin{equation*}
b_{n}	:=	\E e^{aZ_{\tau_{n}^*}(t_{0})} \leq (\E L(t_{0}))^{n} = l(t_{0})^{n} < \infty
\end{equation*}
for any $n\in\N$. Putting
\begin{equation*}
c_{n}	:=	\E \sum_{k=1}^{n}\left(L_{k}(t-S_{\tau_{k-1}^*})-1\right)
\end{equation*}
we have $\sup_{n\ge 1}c_{n}=r^{>}(t_{0})<1$ and thus find by a similar estimation as in the proof of Theorem\;\ref{Thm:shot-noise_positive} for nonnegative $\xi$ that $b_{n}\leq 1+c_{n}b_{n-1}$ and thus $b_{n} \leq (1-r^{>}(t_{0}))^{-1}$ for all $n\in\N$. Hence, $\E e^{aZ(t_{0})}<\infty$, for $Z_{\tau_{n}^*}(t_{0})\uparrow Z(t_{0})$.
\end{proof}

\subsection{Finiteness of Power Moments of Shot-Noise Processes}	\label{subsec:finitieness_power_shot-noise}

Turning to power moments, we consider the case $\xi\ge 0$ a.s.\ only.

\begin{Theorem}	\label{Thm:power_shot-noise_positive}
Let $\xi \geq 0$ a.s. Then for any $p \geq 1$ and $t\in\R$, the following assertions are equivalent:
\begin{align}
\E Z(t)^{p} ~&<~ \infty.	\label{eq:EZ(t)^p<infty}	\\
s_{q}(t) ~:=~	\int \E X(t-y)^{q} \, \mathbb{U}({\rm d}y) ~&<~  \infty	\quad	\text{for all } q \in [1,p];	\label{eq:s_q(t)<infty}
\end{align}
\end{Theorem}
\begin{proof}
``\eqref{eq:EZ(t)^p<infty}$\Rightarrow$\eqref{eq:s_q(t)<infty}'':
Let $\E Z(t)^p < \infty$ and $q \in [1,p]$.
Using the superadditivity of the function $x \mapsto x^q$ for $x \geq 0$, we then infer
\begin{equation*}
\infty  ~>~ \E Z(t)^q ~\geq~ \E \sum_{k \geq 1} X_k(t-S_{k-1})^{q} ~=~ \int_0^\infty \E X(t-y)^{q} \, \mathbb{U}({\rm d}y)
\end{equation*}
which is the desired conclusion.

\vspace{.2cm}
``\eqref{eq:s_q(t)<infty}$\Rightarrow$\eqref{eq:EZ(t)^p<infty}'':
To prove this implication, we write $p=n+\delta$ with $n \in \N_0$, $\delta \in (0,1]$ and use induction on $n$.
When $n=0$, then necessarily $\delta = 1$, \textit{i.e.}, $p=1$. Then there is nothing to verify, for
\begin{equation*}
\E Z(t) ~=~ \int_0^\infty \E X(t-y) \mathbb{U}({\rm d}y) ~=~ s_1(t) ~<~ \infty.
\end{equation*}
In the induction step, we assume that the asserted implication holds for $p=n$
and conclude that it then also holds for $p=n+\delta$ for all $\delta \in (0,1]$.
To this end, assume that $p = n+\delta$ for some $n \in \N$ and $\delta \in (0,1]$ and that $s_q(t) < \infty$ for all $q \in [1,p]$.
By induction hypothesis, $\E Z(t)^n < \infty$.
For $k \in \N$ and $t \in \R$, define
\begin{equation*}
Z_{k}(t) ~:=~ \sum_{j \geq k+1} X_{j}(t-(S_{j-1}-S_{k})).
\end{equation*}
Then $Z_{k}(\cdot)$ is a copy of $Z_{0}(\cdot):=Z(\cdot)$ and also independent of $\mathcal{F}_{k} := \sigma((X_j,\xi_j): j=1,\ldots,k)$.
$Z_{k}$ satisfies $Z_{k}(t)=X_{k+1}(t)+Z_{k+1}(t-\xi_{k+1})$ for all $t \in \R$.
Using \eqref{eq:Gut's_estimate_modified}, we get
\begin{eqnarray*}
Z(t)^{p}
& = &
(X_1(t) + Z_1(t\!-\!\xi_1))^p	\\
& \leq &
X_1(t)^p+Z_1(t\!-\!\xi_1)^p		\\
& & + p2^{p-1}(X_1(t) Z_1(t\!-\!\xi_1)^{p-1} + X_1(t)^nZ_1(t\!-\!\xi_1)^{\delta}).
\end{eqnarray*}
Iterating this inequality and using
\begin{equation}	\label{eq:Z_k(t-S_k)->0}
Z_{k}(t-S_{k}) ~=~ \sum_{j \geq k+1} X_j(t-S_{j-1}) ~\to~ 0 \quad	\text{a.s.\ as } k \to \infty
\end{equation}
we obtain the following upper bound for $Z(t)^p$:
\begin{eqnarray*}
Z(t)^{p}
& \leq &
\sum_{j \geq 1} X_j(t-S_{j-1})^p \\
& & + p2^{p-1}\bigg(\sum_{j \geq 1} X_j(t\!-\!S_{j-1}) Z_j(t\!-\!S_j)^{p-1}	\\
& & \hphantom{ + p2^{p-1}\bigg(} + \sum_{j \geq 1} X_j(t\!-\!S_{j-1})^n Z_j(t\!-\!S_j)^{\delta}\bigg).
\end{eqnarray*}
$\E Z(t)^n < \infty$ implies that $\E Z(t)^q$ is finite for $0 < q \leq n$. Using this and the monotonicity of $Z_j$, we conclude
\begin{equation*}
\E Z(t)^{p}
~\leq~
s_p(t) + p2^{p-1}(s_1(t) \E Z(t)^{p-1} + s_n(t) \E Z(t)^{\delta})
~<~
\infty.
\end{equation*}
\end{proof}

\section{Proofs of the Main Results}	\label{sec:Proofs}

\subsection{Proofs of the Results on a.s.\ Finiteness of $\tau(x)$, $N(x)$ and $\rho(x)$}	\label{subsec:proofs_finiteness}

\begin{proof}[Proof of Theorem\;\ref{Thm:global}]
By Theorem\;2.1 in \cite{GolMal:00}, \eqref{eq:drift_-infty} is necessary and sufficient for $\lim_{n \to \infty} T_n=-\infty$ a.s.\ and thus, by symmetry, \eqref{eq:drift_+infty} is equivalent to
$\lim_{n \to \infty} T_n=\infty$ a.s.
On p.\;1215 of \cite{GolMal:00} it is shown that $\limsup_{n \to \infty}T_n < \infty$ a.s.\ entails $\lim_{n \to \infty} T_n=-\infty$ a.s.
This proves the remaining assertions.
\end{proof}

One half of the proof of Theorem\;\ref{Thm:finiteness_tau} is settled by the following lemma.

\begin{Lemma}   \label{Lem:tau_leq_geometric}
Let $x \in \R$, $\Prob\{\xi<0, \eta \leq x\} = 0$ and $p:=\Prob\{\eta\leq x\}<1$. Then $\Prob\{\tau(x) > n\} \leq p^n$ for $n\in\N$.
If $p=1$, then $\limsup_{n\to\infty}T_{n}=\infty$ a.s.
\end{Lemma}

\begin{proof}
Let $x \in \R$ and $\Prob\{\xi<0, \eta \leq x\} = 0$. Then $p=1$ entails $\xi \ge 0$ a.s., thus $\lim_{n\to\infty}S_n = \infty$ a.s.\ (recalling our standing assumption) and thus, by Theorem\;\ref{Thm:global}, $\limsup_{n \to \infty} T_n = \infty$ a.s.

Now assume that $p<1$. Then $\nu := \inf\{n \in \N: \eta_n > x\}$ has a geometric distribution, namely $\Prob\{\nu > n\} = p^n$ for $n\in\N$. By assumption, $\xi_k \geq 0$ a.s.\ for $k=1,\ldots,n-1$ on $\{\nu=n\}$ whence $T_n = \xi_1 +\ldots + \xi_{n-1} + \eta_n \geq \eta_n > x$ a.s.\ on $\{\nu = n\}$ and therefore
\begin{equation*}
\Prob\{\tau(x) > n\} = \Prob\{T_k \leq x \text{ for } k=1,\ldots,n\}
\leq \Prob\{\nu > n\} = p^n.
\end{equation*}
for any $n\in\N$.
\end{proof}

\subsection{Proofs of the Results on Finiteness of Exponential Moments of $\tau(x)$, $N(x)$ and $\rho(x)$}	\label{subsec:proofs_exponential}

\begin{proof}[Proof of Theorem\;\ref{Thm:finiteness_tau}]
In view of the previous lemma it remains to argue that, given a negatively divergent PRW $(T_{n})_{n \geq 1}$, the a.s.\ finiteness of $\tau(x)$ for some $x\in\R$ implies $\Prob\{\xi<0,\eta\leq x\}=0$ which will be done by contraposition: \\
If $\Prob\{\xi<0,\eta \leq x\}>0$, we can fix $\varepsilon>0$ such that $\Prob\{\xi \leq -\varepsilon, \eta \leq x\}>0$. By negative divergence, $\sup_{n \geq 1} T_n < \infty$ a.s.\ so that we can further pick $y \in \R$ such that $\Prob\{\sup_{n \geq 1} T_n \leq y\} > 0$. Define $m := \inf\{k \in \N_0: k \varepsilon \geq y-x\}$. Then
\begin{align*}
\Prob\{\tau(x) = \infty\} &=\Prob\{\sup_{n \geq 1} T_n \leq x\}    \\
&\geq \Prob \bigg \{\max_{1 \leq k \leq m}\xi_k \leq -\varepsilon,\,\max_{1 \leq k \leq m} \eta_k \leq x, \, \sup_{j>m} T_j - S_m \leq y\bigg\}    \\
&= \Prob\{\xi \leq -\varepsilon, \eta \leq x\}^m\, \Prob\{\sup_{n \geq 1} T_n \leq y\} > 0
\end{align*}
yields the desired conclusion.
\end{proof}

Recall that $\tau^*(x)$ denotes the counterpart of $\tau(x)$ for the ordinary random walk $(S_{n})_{n \geq 0}$ and note also that, for any $a>0$, $\E e^{a\tau^*(x)}<\infty$ for all $x\in\R$ is equivalent to $\E e^{a\tau^*}<\infty$. Put
\begin{equation*}
\nu(x) := \inf\{n\ge 1:\eta_{n}>x\}
\end{equation*}
for $x \in \R$. We make the observation that $\tau^* \wedge \nu(x) \leq \tau(x)$, for
\begin{itemize}
\item[] either $S_{\tau(x)-1}>0$ $(\Rightarrow\tau^*<\tau(x))$,
\item[] or $S_{\tau(x)-1}\leq 0$ and $\eta_{\tau(x)}>x$ $(\Rightarrow\nu(x)\leq \tau(x))$.
\end{itemize}

\begin{Lemma}\label{Lemma:solidarity mgf}
Let $a>0$ and suppose that $\Prob\{\xi<0,\eta\leq x\}>0$ as well as $\E e^{a\tau(x)}<\infty$ for some fixed $x\in\R$. Then $\E e^{a\tau(y)}<\infty$ for all $y\in\R$.
\end{Lemma}

\begin{proof}
By monotonicity, $\E e^{a\tau(y)}<\infty$ for all $y \leq x$.
Now fix some $\varepsilon>0$ such that $\Prob\{\xi\le-\varepsilon,\eta\leq x\}>0$.
Then
\begin{align*}
\E e^{a\tau(x)} ~\geq~ \E e^{a\tau(x+\varepsilon)} \,\Prob\{\xi \leq -\varepsilon, \eta \leq x\}
\end{align*}
implies $\E e^{a\tau(x+\varepsilon)}<\infty$. By repeating this argument with $x+\varepsilon,x+2\varepsilon,\ldots$
and noting that $\Prob\{\xi\le-\varepsilon,\eta\leq x+n\varepsilon\}>0$, we infer $\E e^{a\tau(x+n\varepsilon)}<\infty$ for all $n\in\N$.
\end{proof}

\begin{proof}[Proof of Theorem\;\ref{Thm:exponential_tau}]
(a) If $\Prob\{\xi<0,\eta\leq x\}=0$, then $\tau(y)\leq \nu(y)$ for all $y\leq x$ and therefore $g(y):=\E e^{a\tau(y)}<\infty$ when $e^{a}\Prob\{\eta\leq y\}<1$, because in this case
\begin{align*}
g(y) ~\leq~ \E e^{a\nu(y)} ~=~ \Prob\{\eta>y\} \sum_{n \geq 1} e^{an}\,\Prob\{\eta\leq y\}^{n-1} ~<~ \infty.
\end{align*}
Turning to the asserted equivalence, note first that $e^{a}\Prob\{\xi=0,\eta\leq x\}\ge 1$ implies $\E e^{a\tau(x)}=\infty$ because
\begin{align*}
\sum_{n\ge 1}e^{an}\,\Prob\left\{\max_{1\leq k\leq n}T_{k}\leq x\right\}\ &\ge\ \sum_{n\ge 1}e^{an}\,\Prob\left\{\max_{1\leq k\leq n}T_{k}\leq x,\,S_{n-1}=0\right\}\\
&=~ \Prob\{\eta\leq x\}\sum_{n\ge 1}e^{an}\,\Prob\{\xi=0,\eta\leq x\}^{n-1}~=~ \infty.
\end{align*}
For the converse implication, assume $e^{a} \Prob\{\xi=0,\eta \leq x\} < 1$.
For $n \in \N$, define $\hat{\xi}_n := \xi_n \1_{\{\eta_n \leq x\}} + \1_{\{\eta_n > x\}}$, $n \in \N$. Observe that $\hat{\xi}_n \geq 0$ a.s.\ since $\Prob\{\hat{\xi}_n < 0\} = \Prob\{\xi_n < 0, \eta_n \leq x\} = 0$.
Let $\hat{S}_n := \hat{\xi}_1+\ldots+\hat{\xi}_n$, and $\hat{T}_n := \hat{S}_{n-1}+\eta_n$, $n \in \N$, \textit{i.e.}, $(\hat{T}_n)_{n \geq 1}$ is the PRW based on the sequence $(\hat{\xi}_1,\eta_1),(\hat{\xi}_2,\eta_2),\ldots$. By construction, $\hat{T}_n = T_n$ for all $n \leq \nu(x)$. On the other hand, $\tau(x) \leq \nu(x)$ a.s.\ due to the assumption $\Prob\{\xi<0,\eta \leq x\} = 0$. Consequently, $\hat{\tau}(x) := \inf\{n \geq 1: \hat{T}_n > x\} = \tau(x)$ a.s.
In particular, $\E e^{a \tau(x)}$ is finite iff $\E e^{a\hat{\tau}(x)}$ is finite. To see that the latter is finite in the given situation, let $\hat{\tau}^*(y) := \inf\{n \geq 1: \hat{S}_n > y\}$ for $y \geq 0$ and observe that $\E e^{a \hat{\tau}^*(y)} < \infty$ for all $ y \geq 0$ by Proposition\;1.1 in \cite{IksMei:10b}
since $e^a \Prob\{\hat{\xi}_1 = 0\} = e^a \Prob\{\xi = 0, \eta \leq x\} < 1$.
Pick $u \in \R$ such that $\Prob\{\eta \leq -u\} < e^{-a}$ and define $\nu' := \inf\{n \geq 1: \eta_{\hat{\tau}^*(u+x)+n} > -u\}$.
Then
\begin{equation*}
\E e^{a \nu'} ~=~ \sum_{n \geq 1} e^{an} \Prob\{\nu'=n\}
~=~ \Prob\{\eta > -u\} \sum_{n \geq 1} e^{an} \Prob\{\eta \leq -u\}^{n-1}	~<~	\infty.
\end{equation*}
Since $\hat{S}_n$ is increasing in $n$, we have
$\tau'(x) \leq \hat{\tau}^*(x+u) + \nu'$. Therefore, using the independence of $\hat{\tau}^*(x+u)$ and $\nu'$, we infer
\begin{equation*}
\E e^{a\tau^*(x)}
~\leq~	\E e^{a (\hat{\tau}^*(x+u) + \nu')}
~=~	\E e^{a \hat{\tau}^*(x+u)}  \E e^{a \nu'}
~<~	\infty.
\end{equation*}

\vspace{.2cm}
(b) Since Lemma\;\ref{Lemma:solidarity mgf} gives the equivalence of \eqref{eq:Ee^atau<infty} and \eqref{eq:Ee^atau(y)<infty} and the equivalence of \eqref{eq:E^atau*<infty} and \eqref{eq:-loginf E^-txi>=a} has been shown as Theorem\;1.2 in \cite{IksMei:10b}, we are left with a proof of ``\eqref{eq:Ee^atau(y)<infty}$\Rightarrow$\eqref{eq:-loginf E^-txi>=a}'' and ``\eqref{eq:E^atau*<infty}$\Rightarrow$\eqref{eq:Ee^atau<infty}''.

\vspace{.2cm}
``\eqref{eq:Ee^atau(y)<infty}$\Rightarrow$\eqref{eq:-loginf E^-txi>=a}''
Suppose $\E e^{a\tau(y)}<\infty$ for all $y\in\R$ and recall that $\tau^*\wedge\nu(y)\leq \tau(y)$. Then it follows that
\begin{equation*}
\E\left(e^{a\tau^*}\1_{\{\nu(y)>\tau^*\}}\right) < \infty
\end{equation*}
for all $y\in\R$. Let $\tau_{y}^*$ denote the first strictly ascending ladder epoch of a standard random walk with increment distribution $\Prob\{\xi\in\cdot|\eta\leq y\}$ for any $y$ with $e^{-\theta(y)}:=\Prob\{\eta\leq y\}>0$. Then
\begin{equation*} \Prob\{\tau_{y}^*=n\}~=~ \Prob\{\tau^*=n|\nu(y)>n\} \end{equation*}
for each $n\in\N$ and therefore
\begin{align*}
\infty ~>~ \E\left(e^{a\tau^*}\1_{\{\nu(y)>\tau^*\}}\right) ~&=~ \sum_{n\ge 1}e^{an}\,\Prob\{\tau^*=n,\nu(y)>n\}	\\
&=~ \sum_{n\ge 1}e^{(a-\theta(y))n}\,\Prob\{\tau_{y}^*=n\}	\\
&=~ \E e^{(a-\theta(y))\tau_{y}^*}.
\end{align*}
By invoking Theorem\;1.2 in \cite{IksMei:10b}, we infer
\begin{equation*}
a-\theta(y) ~\leq~ -\log \inf_{t \geq 0} \E(e^{-t\xi}|\eta\leq y)~=~ -\log \inf_{t \geq 0} \E e^{-t\xi} \1_{\{\eta\leq y\}}-\theta(y)
\end{equation*}
and hence $a \leq  -\log \inf_{t \geq 0} \E e^{-t\xi} \1_{\{\eta \leq y\}}$ for all sufficiently large $y$.
It remains to show that $\lim_{y \to \infty} \inf_{t \geq 0} \E e^{-t\xi} \1_{\{\eta \leq y\}} \geq  \inf_{t \geq 0} \E e^{-t\xi}$.
To this end,
put $\varphi_y(t) := \E e^{-t\xi} \1_{\{\eta\leq y\}}$ and notice that, for some $y_0 > 0$, $\Prob\{\xi<0, \eta \leq y_0\} > 0$.
For $y \geq y_0$, $\varphi_y$ assumes its infimum at some unique $0 \leq t_y < \infty$, say.
Let $t_{\infty}$ denote the unique minimizer of $\varphi(t) = \E e^{-t\xi}$ on $[0,\infty)$.
Then, for any given $t > t_{\infty}$, $\varphi(t) > \varphi(t_{\infty})$. (Here, it may happen that $\varphi(t) = \infty$.)
Since $\varphi_y(t) \to \varphi(t)$ and $\varphi_y(t_{\infty}) \to \varphi(t_{\infty})$ as $y \to \infty$,
we find $y_1 \geq y_0$ such that $\varphi_y(t) > \varphi_y(t_{\infty})$ for all $y \geq y_1$.
From the convexity of $\varphi_y$ one can then conclude that $t_y < t$.
Since $t>t_{\infty}$ was chosen arbitrarily, we infer that $\limsup_{y \to \infty} t_y \leq t_{\infty}$.
Analogously, one can show that $\liminf_{y \to \infty} t_y \geq t_{\infty}$.
Consequently, $\lim_{y \to \infty} t_y = t_{\infty}$.
Now we decompose $\varphi_y$ into a decreasing function $\varphi_{y,>}$ and an increasing function $\varphi_{y,\leq}$:
\begin{equation*}
\varphi_y(t)	~=~	\E e^{-t\xi}\1_{\{\xi > 0, \eta \leq y\}} + \E e^{-t\xi} \1_{\{\xi \leq 0, \eta \leq y\}}	~=:~	\varphi_{>,y}(t) + \varphi_{\leq,y}(t),	\quad	t \geq 0.
\end{equation*}
Since $t_y \to t_{\infty}$ as $y \to \infty$, for any given $\delta > 0$ and all large enough $y$, $t_y \in [t_{\infty}-\delta,t_{\infty}+\delta]$.
Therefore, using the monotonicity of $\varphi_{y,>}$ and $\varphi_{y,\leq}$, we obtain
\begin{align*}
\varphi_y(t_y)
& \geq~
\varphi_{>,y}(t_{\infty}+\delta) + \varphi_{\leq,y}(t_{\infty}-\delta)	\\
& \to~
\E e^{-(t_{\infty}+\delta) \xi} \1_{\{\xi > 0\}} + \E e^{-(t_{\infty}-\delta) \xi} \1_{\{\xi \leq 0\}}	\quad &	(y \to \infty)		\\
& \to~
\varphi(t_{\infty})	&	(\delta \to 0)
\end{align*}
where we have used the monotone convergence theorem twice.
This implies \eqref{eq:-loginf E^-txi>=a}.

\vspace{.2cm}
``\eqref{eq:E^atau*<infty}$\Rightarrow$\eqref{eq:Ee^atau<infty}''
Suppose that $\E e^{a\tau^*}<\infty$ and consider the renewal sequence of strictly ascending ladder epochs $(\tau_{n}^*)_{n \geq 0}$ associated with $(S_{n})_{n \geq 0}$, thus $\tau_{1}^*=\tau^*$. Pick $s\in\R$ such that
\begin{equation*}
\gamma := \E\left(e^{a\tau^*} \1_{\{\eta_{\tau^*}\leq \,\xi_{\tau^*}+s\}}\right) < 1
\end{equation*}
and then define $\sigma:=\inf\{n\ge 1:\eta_{\tau_{n}^*}>\xi_{\tau_{n}^*}+s\}$,
which has a geometric distribution on $\N$ with parameter $\Prob\{\eta_{\tau^*}>\xi_{\tau^*}+s\}$.
Since $T_{\tau_{\sigma}^*}=S_{\tau_{\sigma}^*}+\eta_{\tau_{\sigma}^*}-\xi_{\tau_{\sigma}^*}>s$,
we infer $\tau(s)\leq \tau_{\sigma}^*$.
Finally, use that $(\tau_{n}^*-\tau_{n-1}^*,\xi_{\tau_{n}^*},\eta_{\tau_{n}^*})$, $n \in \N$ are i.i.d.\;to infer
\begin{align*}
\E e^{a\tau_{\sigma}^*}\ &=~ \E\left(\prod_{k=1}^{\sigma}e^{a(\tau_{k}^*-\tau_{k-1}^*)}\right)	\\
&=~ \sum_{n \geq 0}\gamma^{n}\,\E\left(e^{a\tau^*}\1_{\{\eta_{\tau^*}>\,\xi_{\tau^*}+s\}}\right) ~<~ \infty
\end{align*}
and therefore $\E e^{a\tau(s)}<\infty$.
If $s \geq x$, this also proves $\E e^{a\tau(x)}<\infty$. Otherwise, consider the level 1 ladder epochs
$\tau_{1}^*(1),\tau_{2}^*(1),\ldots$ of $(S_{n})_{n \geq 0}$ and pick $m$ so large that $s+m\ge x$.
Observe that $\tau(x) \leq \tau_{m}^*(1)+\tau'(s)$ where
\begin{equation*}
\tau'(s):=\inf\{n \geq 1: T_{\tau_{m}^*(1)+n}-S_{\tau_{m}^*(1)} > s\}.
\end{equation*}
Then $\tau'(s)$ is a copy of $\tau(s)$ and independent of $(\tau_{k}^*(1))_{1\leq k\leq m}$.
In combination with $\E e^{a\tau^*(1)}<\infty$, this implies
\begin{equation*}
\E e^{a\tau(x)} ~\leq~ \left(\E e^{a\tau^*(1)}\right)^{m}\E e^{a\tau(s)} ~<~ \infty.
\end{equation*}
The proof is complete.
\end{proof}

\begin{proof}[Proof of Theorem\;\ref{Thm:exponential_N}]
(a) Fix any $a>0$ and $x\in\R$. For $y \geq 0$, define
\begin{equation*}
\widehat{\tau}(y) := \inf\{n\ge 1:\xi_{n}>y\}.
\end{equation*}
Consider the renewal shot-noise process $Z(\cdot)$ with generic response function $X(t):=\sum_{k=1}^{\widehat{\tau}(0)}\1_{\{\eta_{k} \leq t\}}$ and generic renewal increment $\xi':=S_{\widehat{\tau}(0)}>0$ having distribution $\Prob\{\xi\in\cdot|\xi>0\}$.
Then it can be checked that $N(x)=Z(x)$ for all $x\in\R$ and therefore, by Theorem\;\ref{Thm:shot-noise_positive} and Remark \ref{Rem:shot-noise_positive},
that $\E e^{aN(x)}<\infty$ iff
\begin{align}
\int_0^\infty &\Bigg(\E \exp \Bigg(a\sum_{k=1}^{\widehat{\tau}(0)}\1_{\{\eta_k\leq x-u\}} \Bigg)-1\Bigg)\ \mathbb{U}'({\rm d}u) ~<~ \infty,	\label{eq:crucial integral}
\end{align}
where $\mathbb{U}'$ denotes the renewal measure associated with $\xi'$ and satisfies
\begin{equation}\label{eq:renewal fct estimate2}
J_{+}(y) \Prob\{\xi>0\} ~\leq~ \mathbb{U}'(y) ~\leq~  2 J_{+}(y) \Prob\{\xi>0\}
\end{equation}
for all $y > 1$, see \textit{e.g.}\ (4.1) in \cite{Erickson:73}. Since
\begin{equation*}
\E e^{a\1_{\{\eta\leq x\}}}\1_{\{\xi=0\}}=e^{a}\,\Prob\{\xi=0,\eta\leq x\}+\Prob\{\xi=0,\eta>x\},
\end{equation*}
we see that \eqref{eq:e^aP(xi=0,eta<=x)+P(xi=0,eta>x)<1} is equivalent to
\begin{equation}	\label{(2.22)}
\E \exp \Bigg(a \sum_{k=1}^{\widehat{\tau}(0)}\1_{\{\eta_k\leq x\}}\Bigg)
~=~ \frac{\E e^{a\1_{\{\eta\leq x\}}}\1_{\{\xi>0\}}}{1-\E e^{a\1_{\{\eta\leq x\}}}\1_{\{\xi=0\}}}
~<~ \infty
\end{equation}
Validity of \eqref{(2.22)} further implies \eqref{eq:crucial integral} because
\begin{align*}
\int_{0}^{\infty}&\Bigg(\E \exp \Bigg(a\sum_{k=1}^{\widehat{\tau}(0)}\1_{\{\eta_k\leq x-u\}} \Bigg)-1\Bigg)\ \mathbb{U}'({\rm d}u)	\\
&=~ 	\int_{0}^{\infty}\frac{\E e^{a\1_{\{\eta \leq x-u\}}}-1}{1-\E e^{a\1_{\{\eta\leq x-u\}}}\1_{\{\xi=0\}}}\ \mathbb{U}'({\rm d}u)	\\
&=~ 	\int_{0}^{\infty}\frac{(e^{a}-1)\,\Prob\{\eta \leq x-u\}}{1-\E e^{a\1_{\{\eta\leq x-u\}}}\1_{\{\xi=0\}}}\ \mathbb{U}'({\rm d}u)	\\
&\leq~ \frac{(e^{a}-1)}{1-\E e^{a\1_{\{\eta \leq x\}}}\1_{\{\xi=0\}}} \int_{0}^{\infty}\Prob\{(\eta-x)^{-}\ge u\}\ \mathbb{U}'({\rm d}u)	\\
&\leq~	\frac{2 (e^{a}-1) \Prob\{\xi>0\}}{1-\E e^{a\1_{\{\eta \leq x\}}}\1_{\{\xi=0\}}}\,\E J_{+}((\eta-x)^{-})
\end{align*}
where \eqref{eq:renewal fct estimate2} has been utilized for the last line and $\E J_{+}((\eta-x)^{-}) < \infty$ by \eqref{eq:EJ+eta-<infty}.

Since, conversely, \eqref{eq:e^aP(xi=0,eta<=x)+P(xi=0,eta>x)<1} follows directly from \eqref{eq:crucial integral},
we have thus proved the equivalence of \eqref{eq:Ee^aN<infty} and \eqref{eq:e^aP(xi=0,eta<=x)+P(xi=0,eta>x)<1}. To check the remaining assertions is easy and therefore omitted.

\vspace{.2cm}
(c)
First observe that \eqref{eq:Ee^aN*(x)<infty} is equivalent to \eqref{eq:R_geq_a} by Theorem\;1.2 in \cite{IksMei:10b}. Next, we show that $\E e^{aN(x)} < \infty$ for some $x \in \R$ implies $\E e^{aN(x)}<\infty$ for all $x \in \R$. Indeed, since $\Prob\{\xi < 0\} > 0$, for any given $y > x$ we find $n \in \N$ such that $\Prob\{S_n \leq x-y\}>0$ and hence
\begin{eqnarray*}
\infty & > & \E e^{a N(x)}
\ \geq\ \E \1_{\{S_n \leq x-y\}} e^{a \sum_{k>n} \1_{\{T_k-S_n \leq y\}}} \\
& \geq &
\Prob \{S_n \leq x-y\} \, \E e^{a N(y)}.
\end{eqnarray*}
Now we show
``\eqref{eq:Ee^aN(x)<infty}$\Rightarrow$\eqref{eq:Ee^aN*(x)<infty}''.
Since $\Prob\{\xi<0,\eta \leq x\} \to \Prob\{\xi < 0\} > 0$ as $x
\to \infty$, we can choose $x \in \R$ so large such that
$\Prob\{\xi<0,\eta \leq x\} > 0$. Using that $N(x) \geq \tau(x)-1$,
we infer from \eqref{eq:Ee^aN(x)<infty} that $\E e^{a\tau(x)} <
\infty$. By Theorem\;\ref{Thm:exponential_tau}(b), this implies
$\E e^{a \tau^*}< \infty$ which is equivalent to
\eqref{eq:Ee^aN*(x)<infty} by Theorem\;1.2 in \cite{IksMei:10b}.

\noindent\eqref{eq:R_geq_a}$\Rightarrow$\eqref{eq:Ee^aN(x)<infty}.
By \eqref{eq:R_geq_a}, there exists a minimal $\gamma>0$ such that $\E e^{-\gamma\xi}=e^{-a}$.
$\gamma$ can be used to define a new probability measure $\Prob_{\gamma}$ by
\begin{equation}	\label{eq:P_gamma}
\E_\gamma h(S_0,\ldots, S_n)	~=~		e^{an} \E e^{-\gamma S_n}h(S_0,\ldots, S_n),
\quad	n\in \N,
\end{equation}
for each nonnegative Borel measurable function $h$ on $\R^{n+1}$
where $\E_{\gamma}$ denotes the expectation with respect to $\Prob_{\gamma}$.

Recall that $\tau_n^*$ denotes the $n$ strictly increasing ladder index of the process $(S_n)_{n \geq 0}$
and that $\mathbb{U}^>(\cdot) ~:=~ \sum_{n \geq 0} \Prob\{S_{\tau_n^*} \in \cdot\}$ denotes the renewal
measure of the corresponding ladder height process.
Then, according to Theorem\;\ref{Thm:shot-noise_general}
(with  $X(t)=\1_{\{\eta\leq t\}}$) it suffices to prove that
\begin{equation}	\label{eq:r^>(0)<infty}
r^{>}(0) := \int (l(-u)\!-\!1)\,\mathbb{U}^{>}({\mathrm d}u)<\infty,
\end{equation}
where $l(x) :=  \E\left(\prod_{n=1}^{\tau^*}e^{a\1_{\{T_n \leq x\}}}\right)$, $x\in\R$.

For $x\in \R$, set
\begin{equation*}
\beta(x)		~:=~	\sup\{n\leq \tau^*: T_n\leq x\}
\end{equation*}
if $\min_{1 \leq n \leq \tau^*} T_n \leq x$,
and let $\beta(x):=0$, otherwise.
Then $l(x) \leq \E e^{a \beta(x)}$. Therefore, \eqref{eq:r^>(0)<infty} follows from
\begin{equation}    \label{eq:int_e^abeta(x)<infty}
\int \big(\E\exp(a\beta(-u))\!-\!1\big)\,\mathbb{U}^{>}({\mathrm d}u) < \infty.
\end{equation}
Now
\begin{align*}
\E & e^{a\beta(x)}	\\
&~=~
\Prob\{\min_{1 \leq n \leq \tau^*} T_n > x\} + \sum_{n \geq 1} e^{an} \Prob\{\tau^* \geq n, \, T_n\leq x, \min_{n+1\leq
k \leq \tau^*}T_k>x\}	\\
&~\leq~
\Prob\{\min_{1 \leq n \leq \tau^*} T_n>x\} + \sum_{n \geq 1}e^{an}\Prob\{\tau^*\geq n, \, T_n\leq x\}.
\end{align*}
Consequently,
\begin{eqnarray}
\E e^{a\beta(x)}-1
& \leq &
\sum_{n \geq 1} e^{an} \Prob\{\tau^* \geq n, \, T_n\leq x\} -
\Prob\{\min_{1 \leq n \leq \tau^*} T_n\leq x\}	\notag	\\
& \leq &
\sum_{n \geq 1} e^{an} \Prob\{\tau^* \geq n, \, T_n \leq x\}	\notag	\\
& = &
\sum_{n \geq 1} e^{an} \E F(x-S_{n-1}) \1_{\{\tau^* \geq n\}}		\notag	\\
& = &
e^a \sum_{n \geq 0} \E_\gamma e^{\gamma S_{n}} F(x-S_{n}) \1_{\{\tau^* > n\}}	\label{eq:Ee^abeta(x)-1<=}
\end{eqnarray}
where $F(y) := \Prob\{\eta \leq y\}$, $y \in \R$ denotes the distribution function of $\eta$
and where \eqref{eq:P_gamma} has been utilized in the last step.
Let $\sigma^*_0 := 0$ and $\sigma^*_n := \inf\{k > \sigma^*_{n-1}: S_k \leq S_{\sigma^*_{n-1}}\}$ for $n \geq 1$
where $\inf \emptyset = \infty$. We now make use of the following duality, see \textit{e.g.}\ \cite[Theorem\;VIII.2.3(b)]{Asmussen:03},
\begin{equation}	\label{eq:duality_lemma}
\sum_{n \geq 0} \Prob_{\gamma}\{S_n \in \cdot, \tau^* > n\}
~=~
\sum_{n \geq 0} \Prob_{\gamma}\{S_{\sigma^*_n} \in \cdot, \sigma^*_n < \infty\}
\end{equation}
Using this in \eqref{eq:Ee^abeta(x)-1<=} gives
\begin{equation*}
\E e^{a\beta(x)}-1
~\leq~
e^a \sum_{n \geq 0} \E_\gamma e^{\gamma S_{\sigma^*_n}} F(x-S_{\sigma^*_n}) \1_{\{\sigma^*_n  <  \infty\}}.
\end{equation*}
Integrating with $x$ replaced by $-u$ w.r.t.\ $\mathbb{U}^{>}({\mathrm d}u)$ gives
\begin{align}
\int \E & (e^{a\beta(-u)} \!-\!1) \, \mathbb{U}^{>}({\mathrm d}u)		\notag	\\
&\leq~
e^a \int \sum_{n \geq 0} \E_\gamma \bigg[ e^{\gamma S_{\sigma^*_n}} F(-u-S_{\sigma^*_n}) \1_{\{\sigma^*_n  <  \infty\}} \, \mathbb{U}^{>}({\mathrm d}u)\bigg]	\notag	\\
&\leq~
e^a \sum_{n \geq 0} \E_\gamma \bigg[e^{\gamma S_{\sigma^*_n}}  \1_{\{\sigma^*_n  <  \infty\}} \int \mathbb{U}^{>}((z+S_{\sigma^*_n})^-) F(\mathrm{d}z) \bigg]	\notag	\\
&\leq~
e^a \sum_{n \geq 0} \E_\gamma \bigg[e^{\gamma S_{\sigma^*_n}}  \1_{\{\sigma^*_n  <  \infty\}}
\bigg(\int \mathbb{U}^{>}(z^-) F(\mathrm{d}z) +\mathbb{U}^{>}(-S_{\sigma^*_n})\bigg)\bigg]
\label{eq:intEe^abeta(x)-1<=}
\end{align}
where in the last step we have used the subadditivity of $y\mapsto y^-$, $y\in\R$ and $\mathbb{U}^{>}(y)$, $y \geq 0$.
Here, $\int \mathbb{U}^{>}(z^-) F(\mathrm{d}z) = \E \mathbb{U}^{>}(\eta^-)$ is finite due to \eqref{eq:EJ+eta-<infty}
and the fact that $\mathbb{U}^{>}(y) \asymp J_+(y)$ as $y \to \infty$ (see \eqref{eq:asymp}).
Further, again by the subadditivity of $\mathbb{U}^{>}(y)$, we have $\mathbb{U}^{>}(y) = O(y)$ as $y \to \infty$.
In view of this, in order to conclude the finiteness of the series in \eqref{eq:intEe^abeta(x)-1<=} it suffices to show that
\begin{equation}	\label{eq:e^thetayU^<=(dy)<infty}
\sum_{n \geq 0} \E_\gamma e^{\theta S_{\sigma^*_n}}  \1_{\{\sigma^*_n  <  \infty\}} ~<~\infty
\end{equation}
for some $0 < \theta < \gamma$.
It is well known (see \textit{e.g.}\ \cite[Section 2.9]{Gut:09} or \cite[Section 1.4]{Alsmeyer:91}) that
the $\sigma^*_n-\sigma^*_{n-1}$, $n \in \N$ are i.i.d.\;random variables taking values in $\N \cup \{\infty\}$.
In particular, $\Prob_{\gamma}\{\sigma^*_n < \infty\} = \Prob_{\gamma}\{\sigma^*_1 < \infty\}^n$ (see, for instance, Theorem\;1.4.3 in \cite{Alsmeyer:91}).
Hence, when $\Prob_{\gamma}\{\sigma^*_1 < \infty\} < 1$, then
\begin{eqnarray*}
\sum_{n \geq 0} \E_\gamma e^{\theta S_{\sigma^*_n}}  \1_{\{\sigma^*_n  <  \infty\}}
& \leq &
\sum_{n \geq 0} \Prob_\gamma \{\sigma^*_n  <  \infty\}	\\
& = &
\sum_{n \geq 0} \Prob_\gamma \{\sigma^*_1  <  \infty\}^n
~=~	\frac{1}{\Prob_{\gamma}\{\sigma^*_1 = \infty\}}	~<~	\infty
\end{eqnarray*}
where the first inequality follows from the fact that $S_{\sigma^*_n} \leq 0$ on $\{\sigma^*_n  <  \infty\}$.
If, on the other hand, $\Prob_{\gamma}{\{\sigma^*_1 <  \infty\}} =  1$, then we can drop the indicators in \eqref{eq:e^thetayU^<=(dy)<infty} and get
$\E_\gamma e^{\theta S_{\sigma^*_n}}  \1_{\{\sigma^*_n  <  \infty\}} = \E_\gamma e^{\theta S_{\sigma^*_n}}$.
By the discussion in \cite[Section 2.9]{Gut:09} or \cite[Section 1.4]{Alsmeyer:91}, the $S_{\sigma^*_n}-S_{\sigma^*_{n-1}}$, $n \in \N$ are i.i.d.\;random variables.
Therefore, $\E_\gamma e^{\theta S_{\sigma^*_n}} = (\E_{\gamma} e^{\theta S_{\sigma^*_1}})^n$ for each $n \in \N$.
By the definition of $\sigma^*_1$, $\Prob_{\gamma} \{S_{\sigma^*_1} \leq 0\} = 1$.
What is more, we have $\Prob_{\gamma} \{S_{\sigma^*_1} < 0\} > 0$ since $\Prob\{\xi < 0\} > 0$ by assumption.
Consequently, $\E_{\gamma} e^{\theta S_{\sigma^*_1}} < 1$.
From these facts, we can derive the convergence of the series in \eqref{eq:e^thetayU^<=(dy)<infty}:
\begin{equation*}
\sum_{n \geq 0} \E_\gamma e^{\theta S_{\sigma^*_n}}  ~=~	\sum_{n \geq 0} (\E_\gamma e^{\theta S_{\sigma^*_1}})^n
~=~ \frac{1}{1-\E_\gamma e^{\theta S_{\sigma^*_1}}}
~<~	\infty.
\end{equation*}
The proof is complete.
\end{proof}

\begin{proof}[Proof of Theorem\;\ref{Thm:exponential_rho}]
This proof is based on the two inequalities
\begin{equation}    \label{eq:rho(x)=n_<=_Tn}
\Prob\{\rho(x)=n\} \leq \Prob\{T_n \leq x\}, \quad x\in\R
\end{equation}
and
\begin{equation}    \label{eq:rho(x)_>=_n_>=_Tn}
\Prob\{\rho(x) \geq n\} \geq \Prob\{T_n\leq x\}, \quad x\in\R.
\end{equation}
We can write $\E e^{a\rho(x)}$ in the following two ways:
\begin{eqnarray}
\E e^{a \rho(x)}
& = &
\sum_{n \geq 0} e^{an} \Prob\{\rho(x)=n\}    \label{eq:e^arho(x)1}   \\
& = &
e^{-a}\left((e^a-1)\sum_{n \geq 0} e ^{an}\Prob\{\rho(x) \geq n\} + 1\right).   \label{eq:e^arho(x)2}
\end{eqnarray}
The implications ``\eqref{eq:V_a(y)<infty}$\Rightarrow$\eqref{eq:Ee^arho<infty}''
and ``\eqref{eq:V_a(x)<infty_fax}$\Rightarrow$\eqref{eq:Ee^arho(x)<infty_fax}'' follow from \eqref{eq:rho(x)=n_<=_Tn} and
\eqref{eq:e^arho(x)1}. In turn, the implications ``\eqref{eq:Ee^arho<infty}$\Rightarrow$\eqref{eq:V_a(y)<infty}'' (for fixed $y=x$)
and ``\eqref{eq:Ee^arho(x)<infty_fax}$\Rightarrow$\eqref{eq:V_a(x)<infty_fax}'' follow from \eqref{eq:rho(x)_>=_n_>=_Tn} and
\eqref{eq:e^arho(x)2}.  

Next assume that \eqref{eq:a<-logP(xi=0)_Ee^-gammaeta<infty} holds in case $\Prob\{\xi \geq 0\}=1$
and that \eqref{eq:a<=R_extra} holds in case $\Prob\{\xi < 0\} > 0$.
Then, by Proposition\;\ref{Prop:IksMei:10}, $V_a^*(y) ~:=~ \sum_{n \geq 0} e^{an}\Prob\{S_n \leq y\}$ is finite for all $y\in\R$ and
$V_a^*(y) \leq C e^{\gamma y}$ for some constant $C>0$ and all $y \geq 0$. Further, by assumption, $\E e^{-\gamma \eta} < \infty$. Taking all this into account, we infer, for $x\in\R$ (condition $\Prob\{\eta\leq x\}>0$ is not required),
\begin{eqnarray*}
V_a(x) = e^a \E V_a^* (x-\eta) \leq e^a V_a^*(0) + e^a C e^{\gamma x} \E e^{-\gamma \eta} < \infty.
\end{eqnarray*}
Thus, the implications ``\eqref{eq:a<-logP(xi=0)_Ee^-gammaeta<infty}$\Rightarrow$\eqref{eq:V_a(y)<infty}''
and ``\eqref{eq:a<=R_extra}$\Rightarrow$\eqref{eq:V_a(x)<infty_fax}'' hold.

We are now left with the proofs of the implications ``\eqref{eq:V_a(y)<infty}$\Rightarrow$\eqref{eq:a<-logP(xi=0)_Ee^-gammaeta<infty}'' and ``\eqref{eq:V_a(x)<infty_fax}$\Rightarrow$\eqref{eq:a<=R_extra}''. Assume that  \eqref{eq:V_a(y)<infty} holds. We have to show that $a < -\log \beta$ with $\beta:=\Prob\{\xi=0\}$. This is trivial in case $\beta = 0$, and is a consequence of the chain of inequalities
\begin{eqnarray*}
\infty
& > &
\sum_{n \geq 1} e^{an} \Prob\{T_n \leq x\}  \\
& \geq &
\sum_{n \geq 1} e^{an} \Prob\{\xi_1 = \ldots = \xi_{n-1} = 0, \eta_n \leq x\}   \\
& = &
e^{a}  \Prob\{\eta \leq x\} \sum_{n \geq 0} (\beta e^a)^n
\end{eqnarray*}
in case $\beta \in (0,1)$, since $\Prob\{\eta \leq x\} > 0$ by the assumption. The inequality $\E e^{-\gamma\eta}<\infty$ will be established at the end of the proof.

Assume now that $\Prob\{\xi < 0\} > 0$ and that \eqref{eq:V_a(x)<infty_fax} holds. Thus,
\begin{equation}    \label{eq:EV^*_a<infty}
\infty > \sum_{n \geq 1} e^{an} \Prob\{T_n \leq x\} = e^a \E V^*_a(x-\eta).
\end{equation}
In particular, $V^*_a(y)$ is finite for some $y \in \R$.
This yields $a < R$ or $a=R$ and $\E \xi e^{-\gamma \xi} < 0$ in view of Proposition\;\ref{Prop:IksMei:10}.

It remains to prove that $\E e^{-\gamma \eta} < \infty$ under the assumption \eqref{eq:V_a(y)<infty} as well as under the assumption \eqref{eq:V_a(x)<infty_fax}. To this end, notice that by what we have already shown, in both cases, $V^*_a(y)$ is finite for all $y \in \R$ and $0 < c := \inf_{y \geq 0} e^{-\gamma y} V^*_a(y) < \infty$ by Proposition\;\ref{Prop:IksMei:10}. Thus, in view of \eqref{eq:EV^*_a<infty}, we obtain 
\begin{equation*}
\infty > \E V^*_a(x-\eta)
\geq c e^{\gamma x} \E e^{- \gamma \eta} \1_{\{\eta \leq x\}},
\end{equation*}
which immediately leads to the conclusion that $\E e^{-\gamma \eta} < \infty$. The proof is herewith complete.
\end{proof}

\subsection{Proofs of the Results on Finiteness of Power Moments of $N(x)$ and $\rho(x)$}	\label{subsec:proofs_power}

\begin{proof}[Proof of Theorem\;\ref{Thm:power_N}]
Assume first that $\xi \geq 0$ a.s.\ and fix an arbitrary $x \in \R$.
According to parts (a) and (b) of Theorem\;\ref{Thm:exponential_N}, whenever
$N(x)<\infty$ a.s.\ it has some finite exponential moments.
In particular, $\E N(x)^p<\infty$ for every $p>0$.
Therefore, from now on, we assume that $\Prob\{\xi<0\}>0$.

\noindent
``\eqref{eq:EN*^p<infty}$\Leftrightarrow$\eqref{eq:criterion_EN*^p<infty}'':
To prove this equivalence, it suffices to show that
$\E N^*(x)^p < \infty$ iff $\E J_+(\xi^-)^{p+1} < \infty$.
This follows from the discussion on p.\;27 in \cite{Kesten+Maller:96}.


\noindent
``\eqref{eq:EN*^p<infty},\eqref{eq:criterion_EN*^p<infty}$\Rightarrow$ \eqref{eq:EN^p<infty}'':
For any $x \in \R$, $\E J_+(\eta^-)<\infty$ is equivalent to $\E J_+((\eta-x)^-)<\infty$.
Further, (by the equivalence \eqref{eq:EN*^p<infty}$\Leftrightarrow$\eqref{eq:criterion_EN*^p<infty}) we know
that $\E N^*(x)^p < \infty$ for some $x \geq 0$ implies $\E N^*(x)^p < \infty$ for all $x \geq 0$.
Thus replacing $\eta$ by $\eta-x$ it suffices to prove that
$\E N(0)^p<\infty$ if $\E N^*(0)^p < \infty$ and $\E J_+(\eta^-)<\infty$.

\noindent
\underline{Case 1:} $p \in (0,1)$.
Using the subadditivity of the function $x \mapsto x^p$, $x \geq 0$ we obtain
\begin{eqnarray*}
N(0)^p
& \leq &
\bigg(\sum_{ k \geq 1} \1_{\{T_k\leq 0,\, S_{k-1} \leq 0\}}\bigg)^p
+ \bigg(\sum_{k \geq 1}\1_{\{T_k\leq 0,\, S_{k-1}>0\}}\bigg)^p	\\
& \leq &
N^*(0)^p+ \sum_{k \geq 1}\1_{\{0< S_{k-1} \leq \eta_k^-\}}
\quad	\text{a.s.}
\end{eqnarray*}
Since $\E N^*(0)^p<\infty$ by assumption, it remains to check that
\begin{equation}		\label{eq:EU(eta^-]<infty}
\sum_{k \geq 1}\Prob\{0<S_{k-1} \leq \eta_k^-\}
~<~
\infty.
\end{equation}
$\lim_{n \to \infty} T_n = \infty$ implies $\lim_{n \to \infty} S_n=+\infty$ a.s.
The latter ensures $\E \tau^* < \infty$.
Let $\mathbb{U}^>(\cdot)$ be the renewal function of the renewal process of strict ladder heights.
For $x \geq 0$ we have
\begin{eqnarray*}
\sum_{k \geq 1} \Prob\{0 < S_{k-1} \leq x\}
& = &
\int_0^\infty \E  \bigg(\sum_{k=0}^{\tau^*-1}\1_{\{-y<S_k\leq x-y\}}\bigg) \, {\rm d}\mathbb{U}^>(y)	\\
& = &
\E \sum_{k=0}^{\tau^*-1}\bigg(\mathbb{U}^>(x-S_k)-\mathbb{U}^>(-S_k)\bigg)
~\leq~
\E \tau^* \mathbb{U}^>(x),
\end{eqnarray*}
where in the last step the subadditivity of the function $x \mapsto \mathbb{U}^>(x)$, $x \geq 0$ has been utilized.
Now \eqref{eq:EU(eta^-]<infty} follows from the last inequality, the fact that
\begin{equation}	\label{eq:U>(x)_asymp_J_+(x)}
\mathbb{U}^>(x) \asymp J_+(x)	\quad	\text{as } x \to \infty
\end{equation}
(see \eqref{eq:asymp}) and the assumption $\E J_+(\eta^-)<\infty$.

\noindent
\underline{Case 2:} $p \geq 1$.
According to \cite[Theorem\;2.1 and formulae (2.9) and
(2.10)]{Kesten+Maller:96}, the first two conditions in
\eqref{eq:criterion_EN*^p<infty} imply
\begin{equation}    \label{eq:Etau*^p+1<infty}
\E (\tau^*)^{p+1} < \infty.
\end{equation}
Let $\tau_0^* := 0$ and $\tau_n^* := \inf\{k > \tau_{n-1}^*: S_k > S_{\tau_{n-1}^*}\}$.
Retaining the notation of Sect.\;\ref{sec:shot-noise} let
$X_n(x) := \sum_{k=\tau_{n-1}^*+1}^{\tau_n^*} \1_{\{T_k \leq x\}}$ and
$\xi_n := S_{\tau_n^*}-S_{\tau_{n-1}^*}$ and observe that $Z(x)=N(x)$.
Since the so defined $\xi_n$ are a.s.\ positive,
we can apply Theorem\;\ref{Thm:power_shot-noise_positive} to conclude that it is enough to show that,
for every $q \in [1,p]$,
\begin{equation}\label{eq:special_s_q(0)<infty}
\int_0^\infty \E \bigg(\sum_{k=1}^{\tau^*}\1_{\{T_k\leq -y\}}
\bigg)^q \, {\rm d} \mathbb{U}^>(y)<\infty,
\end{equation}
where, as above, $\mathbb{U}^>(\cdot)$ is the renewal function of $(S_{\tau_n^*})_{n \geq 0}$.
Fix any $q \in [1,p]$.
For $x \leq 0$, it holds that
\begin{align*}
\bigg(\sum_{k=1}^{\tau^*} & \1_{\{T_k\leq x\}}\bigg)^q	\\
&\leq~
\bigg(\sum_{k=1}^{\tau^*}\big(\1_{\{S_{k-1}-\eta_k^-\leq x,\, -\eta^-_k\leq x\}}
+ \1_{\{S_{k-1}-\eta_k^-\leq x,\,-\eta^-_k>x\}}  \big)\bigg)^q	\\
&\leq~
2^{q-1}\bigg(\bigg(\sum_{k=1}^{\tau^*} \1_{\{-\eta^-_k\leq x\}}\bigg)^q
+ \bigg(\sum_{k=1}^{\tau^*} \1_{\{S_{k-1}-\eta_k^-\leq x, \,-\eta^-_k> x\}}\bigg)^q \bigg)	\\
&=:~	2^{q-1}(I_1(x)+I_2(x)).
\end{align*}
By \cite[Theorem\;5.2 on p.\;24]{Gut:09}, there exists a positive
constant $B_q$ such that
\begin{eqnarray*}
\E \int_0^\infty I_1(-y) \, {\rm d} \mathbb{U}^>(y)
& \leq &
B_q \E (\tau^*)^q \int_0^\infty \Prob\{\eta^- \geq y\} \, {\rm d}\mathbb{U}^>(y)	\\
& \leq &
B_q \E (\tau^*)^q \E \mathbb{U}^>(\eta^-).
\end{eqnarray*}
Here, $\E \mathbb{U}^>(\eta^-)<\infty$ in view of \eqref{eq:U>(x)_asymp_J_+(x)}
and the last condition in \eqref{eq:criterion_EN*^p<infty}. 
$\E (\tau^*)^q<\infty$ is a consequence of \eqref{eq:Etau*^p+1<infty}.

Turning to the term involving $I_2$, notice that from the inequality $(x_1+\ldots+x_m)^q \leq m^{q-1} (x_1^q+\ldots+x_m^q)$, $x_1, \ldots,x_m \geq 0$
and the subadditivity of the function $x \mapsto \mathbb{U}^>(x)$, $x \geq 0$
it follows that
\begin{eqnarray*}
\int_0^\infty I_2(-y) \, {\rm d} \mathbb{U}^>(y)
& \leq &
(\tau^*)^{q-1}\sum_{k=1}^{\tau^*} \int_0^\infty
\1_{\{S_{k-1}-\eta_k^-\leq -y,\,-\eta_k^->-y\}} \, {\rm d} \mathbb{U}^>(y)	\\
& = &
(\tau^*)^{q-1}\sum_{k=1}^{\tau^*}(\mathbb{U}^>(\eta_k^--S_{k-1})-\mathbb{U}^>(\eta_k^-))	\\
& \leq &
(\tau^*)^{q-1}\sum_{k=0}^{\tau^*-1}\mathbb{U}^>(-S_{k})	\\
& \leq &
(\tau^*)^{q-1}\sum_{k=0}^{\tau^*-1}\mathbb{U}^>(\xi_1^-+\ldots+\xi_k^-)	\\
& \leq &
(\tau^*)^{q-1}\bigg(1+
\sum_{k=1}^{\tau^*-1} \big(\mathbb{U}^>(\xi_1^-)+\ldots+\mathbb{U}^>(\xi_k^-)\big)\bigg)	\\
& = &
(\tau^*)^{q-1}\bigg(1+
\sum_{k=1}^{\tau^*-1}(\tau^*-k)\mathbb{U}^>(\xi_k^-)\bigg)	\\
& \leq &
(\tau^*)^{q-1}+(\tau^*)^q \sum_{k=1}^{\tau^*} \mathbb{U}^>(\xi_k^-).
\end{eqnarray*}
By H\"{o}lder's inequality,
\begin{equation*}
\E (\tau^*)^q \sum_{k=1}^{\tau^*} \mathbb{U}^>(\xi_k^-)
\leq (\E (\tau^*)^{q+1})^{q/(q+1)}
\bigg(\E \bigg(\sum_{k=1}^{\tau^*} \mathbb{U}^>(\xi_k^-)\bigg)^{q+1}\bigg)^{1/(q+1)}.
\end{equation*}
The finiteness of the
first factor is secured by \eqref{eq:Etau*^p+1<infty}. According to
\cite[Theorem\;5.2 on p.\;24]{Gut:09}, the second factor is finite
provided $\E (\tau^*)^{q+1}<\infty$ and $\E
\mathbb{U}^>(\xi^-)^{q+1}<\infty$. The former follows from \eqref{eq:Etau*^p+1<infty}
the latter from \eqref{eq:U>(x)_asymp_J_+(x)} and \eqref{eq:criterion_EN*^p<infty}.
Thus we have proved that $\E \int_0^\infty I_2(-y){\rm d}\,\mathbb{U}^>(y)<\infty$,
hence \eqref{eq:special_s_q(0)<infty}.

\noindent
``\eqref{eq:EN^p<infty}$\Rightarrow$\eqref{eq:EN*^p<infty}'':
Assume that $\E N(x)^p<\infty$.
We only have to prove that $\E N^*(y)^p<\infty$ for some $y \geq 0$.

\noindent
\underline{Case 1:} $p\in (0,1)$.
By \cite[Theorem\;2]{Hitc:88}, without loss of generality, we can assume that $\xi$ and $\eta$ are independent.
We will briefly explain how this reduction can be justified.
Let $(\eta_n')_{n \in \N}$ be a sequence of i.i.d.\;copies of $\eta$ and assume that this sequence is independent of the sequence $((\xi_n,\eta_n))_{n \in \N}$.
Define $T_n' := S_{n-1}+\eta_n'$, $n \in \N$ and $\mathcal{F}'_n := \sigma((\xi_k,\eta_k), \eta_k': k = 1,\ldots,n)$. Then
\begin{equation*}
\Prob(T_n \leq x|\mathcal{F}'_{n-1}) = \Prob(\eta_n \leq x-S_{n-1}|\mathcal{F}'_{n-1}) = G(x-S_{n-1}) \quad   \text{a.s.}
\end{equation*}
where $G(t) := \Prob\{\eta \leq t\}$, $t \in \R$
and, analogously,
\begin{equation*}
\Prob(T_n' \leq x|\mathcal{F}'_{n-1}) = \Prob(\eta_n' \leq x-S_{n-1}|\mathcal{F}'_{n-1}) = G(x-S_{n-1})   \quad   \text{a.s.},
\end{equation*}
that is, the sequences $(\1_{\{T_n \leq x\}})_{n \in \N}$ and
$(\1_{\{T_n' \leq x\}})_{n\in\N}$ are tangent.
Moreover, $(\xi_k)_{k \in \N}$ and $(\eta'_k)_{k \in \N}$ are independent.
This means that we may work under the additional assumption of
independence between the random walk and the perturbating
sequence. In the following, we do not introduce a new notation to
indicate this feature.

Let $y \geq x$ be such that $\Prob\{\eta \leq y\} > 0$ and let
$A:=\{N^*(x-y)>0\}$. Observe that $\Prob(A) > 0$ since we assume
that $\Prob\{\xi < 0\} > 0$. The following inequality holds
a.s.\;on $A$:
\begin{eqnarray*}
N(x)^p
& \geq &
\bigg( \sum_{k \geq 1}\1_{\{S_{k-1} \leq x-y,\,\eta_k\leq y\}}\bigg)^p   \\
& = &
N^*(x-y)^p\bigg(\sum_{k \geq 1} \1_{\{S_{k-1} \leq x-y\}} \1_{\{\eta_k \leq y\}} / N^{*}(x-y) \bigg)^p    \\
& \geq &
N^*(x-y)^{p-1} \sum_{k \geq 1} \1_{\{S_{k-1} \leq x-y\}} \1_{\{\eta_k\leq y\}},
\end{eqnarray*}
where for the second inequality the concavity of $t \mapsto t^p$, $t \geq 0$ has been used.
Taking expectations gives
\begin{eqnarray*}
\infty
& > &
\E N(x)^p
~\geq~
\E \Bigg(\1_A N^*(x-y)^{p-1} \sum_{k \geq 1} \1_{\{S_{k-1} \leq x-y\}} \1_{\{\eta_k\leq y\}}\Bigg)  \\
& = &
\Prob\{\eta \leq y\} \, \E N^*(x-y)^p.
\end{eqnarray*}
An appeal to Lemma\;\ref{Lem:U_p_local_finiteness} completes the
proof of this case.

\noindent
\underline{Case 2:} $p\geq 1$.
It holds that
\begin{eqnarray*}
\infty
&>&
\E N(x)^p \ \geq \ \E\bigg( \sum_{k \geq 1}\1_{\{S_{k-1}\leq x-y,\,\eta_k\leq y\}}\bigg)^p \\
&\geq&
{\rm const}\,\E N^*(x-y)^p (\Prob\{\eta\leq y\})^p,
\end{eqnarray*}
where at the last step the convex function inequality
\cite[Theorem\;3.2]{BurkDG:72}, applied to $\Phi(t)=t^p$, has been
utilized. An appeal to Lemma\;\ref{Lem:U_p_local_finiteness}
completes the proof.
\end{proof}

Turning to the proof of Theorem\;\ref{Thm:power_rho}, we start with a simple lemma.

\begin{Lemma}   \label{Lem:rho_trivial}
Let $x \in \R$. Then the following assertions are equivalent:
\begin{itemize}
    \item[(i)]		$\rho(x) = 0$ a.s.;
    \item[(ii)]		$\inf_{k \geq 1} T_k>x$ a.s.;
    \item[(iii)]	$\Prob\{\xi \geq 0\} = 1$ and $\Prob\{\eta > x\} = 1$.
\end{itemize}
\end{Lemma}
\begin{proof}
The equivalence of (i) and (ii) follows just from the definition
of $\rho(x)$.

If (iii) holds, then $T_n
> x$ a.s.\;for all $n \in \N$, which is equivalent to (ii).
Conversely, since
$\{\eta_1 \leq x\} \subseteq \{\inf_{n \geq 1}T_n \leq x\}$,
the condition $\eta > x$ a.s.\;is necessary for (ii) to hold.
It remains to show that the condition $\xi \geq 0$ a.s.\;is also
necessary for (ii) to hold. To this end,
assume that $\Prob\{\xi_1 \leq -\varepsilon\} > 0$ for some $\varepsilon > 0$.
Further pick $y \in \R$ with $\Prob\{\eta \leq y\}>0$ and choose $n$ so large such that $y-n\varepsilon \leq x$.
Then
\begin{eqnarray*}
\Prob\{\underset{k \geq 1}{\inf}\,T_k\leq x\}
& \geq &
\Prob\{T_{n+1} \leq x\} \geq \Prob\{T_{n+1} \leq y-n\varepsilon\}    \\
& \geq &
\Prob\{\xi_1 \leq -\varepsilon,\ldots,\xi_n \leq -\varepsilon, \eta_{n+1} \leq y\} > 0,
\end{eqnarray*}
which completes the proof.
\end{proof}

\begin{proof}[Proof of Theorem\;\ref{Thm:power_rho}]
``\eqref{eq:Erho*^p<infty}$\Leftrightarrow$\eqref{eq:criterion_Erho^p<infty}''
was proved in \cite[Theorem\;2.1 and formulae (2.9) and (2.10)]{Kesten+Maller:96}.

From the representation
\begin{eqnarray*}
\E \rho(x)^p
&=&
\sum_{n \geq1} n^p \Prob\{\rho(x)=n\} \\
&=&
\sum_{n \geq 1} n^p \Prob\{T_n\leq x, \inf_{k \geq n+1} T_k>x\}   \\
&=&
\sum_{n \geq 1} n^p \big(\Prob\{\inf_{k \geq n}T_k \leq x\} - \Prob\{\inf_{k \geq n+1}T_k\leq x\}\big),
\end{eqnarray*}
and Lemma\;\ref{Lem:summation_by_parts} it follows that $\E \rho(x)^p < \infty$ iff
\begin{equation}    \label{eq:power_renewal_function}
\sum_{n \geq 1} n^{p-1} \Prob\{\inf_{k \geq n} T_k \leq x\}	~=~	\E \mathbb{U}_{p-1}(x-\inf_{k \geq 1} T_k) ~<~ \infty,
\end{equation}
where $\mathbb{U}_{p-1}(y):=\sum_{n \geq 0}n^{p-1} \Prob\{S_n \! \leq \! y\}$ is the power renewal function of $(S_n)_{n \geq 0}$ at $y \in \R$.
Indeed, with $b_n = \Prob\{\inf_{k \geq n} T_k \leq x\} - \Prob\{\inf_{k \geq n+1} T_k \leq x\}$ in Lemma\;\ref{Lem:summation_by_parts}
we have $\sum_{k=n}^\infty b_k = \Prob\{\inf_{k \geq n}T_k \leq x\}$, since $\lim_{n \to \infty} \Prob\{\inf_{k \geq n} T_k\leq x\}=0$
due to the assumption that $T_n \to \infty$ a.s.

\noindent
``\eqref{eq:Erho^p<infty}$\Rightarrow$\eqref{eq:Erho*^p<infty}'':
Suppose \eqref{eq:power_renewal_function} holds for some $x \in \R$.
We distinguish two cases.

\noindent
\underline{Case 1:} $\Prob\{\inf_{k \geq 1} T_k>x\} = 1$.
By Lemma\;\ref{Lem:rho_trivial}, the condition $\inf_{k \geq 1}
T_k>x$ a.s.\ is equivalent to $\Prob\{\xi\geq 0, \eta>x\}=1$.
Hence, \eqref{eq:criterion_Erho^p<infty} trivially holds and,
since ``\eqref{eq:Erho*^p<infty}$\Leftrightarrow$\eqref{eq:criterion_Erho^p<infty}'' has already been established,
also \eqref{eq:Erho*^p<infty}.

\noindent
\underline{Case 2:} $\Prob\{{\inf}_{k \geq 1}T_k>x\}<1$.
In this case, $\mathbb{U}_{p-1}(y)$ must be finite for some $y\geq 0$. From
\cite[Theorem\;2.1]{Kesten+Maller:96}, we infer that $\mathbb{U}_{p-1}(y) <
\infty$ and $\E \rho^*(y)^p<\infty$ for all $y \geq 0$.
Further, by \eqref{eq:asymp},
$\mathbb{U}_{p-1}(y) \asymp J_+(y)^p$ as $y\to\infty$. 
Consequently, since for
any fixed $z\in\R$, $J_+(y+z)^p \sim J_+(y)^p$ as $y\to\infty$,
\eqref{eq:power_renewal_function} implies that $\E
J_+((\inf_{k \geq 1} T_k)^-)^p<\infty$.\label{sas} From
\begin{equation*}
T_k ~=~ \xi_1+\ldots+\xi_{k-1}+\eta_k ~\leq~ \xi_1^++\ldots \xi_{k-1}^++\eta_k ~=:~	\widehat{T}_k,
\quad k \in \N,
\end{equation*}
we conclude that also $\E J_+((\inf_{k \geq 1}\widehat{T}_k)^-)^p<\infty$.
Thus, it suffices to show that $\E J_+((\inf_{k \geq1} \widehat{T}_k)^-)^p<\infty$ implies
$\E J_+(\eta^-)^{p+1} < \infty$. To a large extent, this follows from the proof of \cite[Lemma\;3.4]{AlsIks:09},
although some details have to be explained.

Pick $\varepsilon>0$ such that $\alpha := \Prob\{\inf_{k \geq 1} \widehat{T}_k \geq -\varepsilon\}>0$.
Such an $\varepsilon$ exists since we assume that $T_n \to \infty$ a.s.
Let $(M_k,Q_k)$, $k \geq 1$ be independent copies of a random vector
$(M,Q) := (e^{-\xi^+},e^{-\eta})$, and set
\begin{equation*}
\Pi_k ~:=~ e^{-(\xi_1^++\ldots + \xi_{k}^+)} ~=~ \prod_{j=1}^k M_j,	\quad	k \in \N_0.
\end{equation*}
Using this notation the function $J$ defined after (2) in \cite{AlsIks:09} coincides with the function $J_+$ defined after \eqref{eq:A(x)} if we use the convention that $J_+(x) = 0$ for $x < 0$.
In the cited work it was proved that, for $\delta>\varepsilon$ and for every
nondecreasing and absolutely continuous function $f: [0,\infty) \to [0,\infty)$,
we have
\begin{equation}	\label{eq:AlsIks:09}
\E f\bigg( \sup_{k \geq 1} \Pi_{k-1} Q_k \bigg)
~\geq~ \alpha \E \bigg( \1_{\{Q > e^{2 \delta}\}} f(Q^{1/2}) J_+\bigg(\frac{\log Q}{2} \bigg)\bigg).
\end{equation}
The idea now is to choose $f(x) :=J_+(\log^+ x)^p$ for $x > 0$, $f(0)=0$.
Then \eqref{eq:AlsIks:09} becomes
\begin{eqnarray*}
\E J_+\bigg(\bigg[\inf_{k \geq 1} \widehat{T}_k\bigg]^-\bigg)^p
& \geq &
\alpha \E \big(\1_{\{\eta < -2 \delta\}} J_+(\eta^-/2)^p J_+(\eta^-/2)\big)	\\
& \geq &
\alpha 2^{-(p+1)} \E (\1_{\{\eta < -2 \delta\}} J_+(\eta^-)^{p+1})
\end{eqnarray*}
where in the last step, we have used that, by Lemma\;\ref{Lem:J_+}(c), $J_+(x/2) \geq 2^{-1} J_+(x)$ for all $x \geq 0$.
So in order to make the argument rigorous it remains to show that $f$ has the properties needed.
The latter follows from Lemma\;\ref{Lem:J_+}.

\noindent
``\eqref{eq:Erho*^p<infty}$\Rightarrow$\eqref{eq:Erho^p<infty}'':
We have to prove that the inequality in \eqref{eq:power_renewal_function}
holds for any $x \in \R$. By \cite[Theorem\;2.1]{Kesten+Maller:96},
$\E \rho^*(y)^p < \infty$ for some $y \geq 0$ ensures
that $\mathbb{U}_{p-1}(y) < \infty$ for every $y \geq 0$,
and by \eqref{eq:asymp}, $\mathbb{U}_{p-1}(y)\asymp J_+(y)^p$, $y\to\infty$.

\noindent
\underline{Case 1:}
There exists $y \in \R$ such that $\inf_{k \geq 1}T_k > y$ a.s. Then
\begin{equation*}
\E \mathbb{U}_{p-1}(x-\inf_{k \geq 1} T_k)	~\leq~	\mathbb{U}_{p-1}(x-y)	~<~	\infty,
\end{equation*}
and \eqref{eq:power_renewal_function} holds.

\noindent
\underline{Case 2:}
$\Prob\{\inf_{k \geq 1}T_k>y\}<1$ for all $y \in \R$.
To guarantee that the inequality \eqref{eq:power_renewal_function}
holds it remains to prove that
$\E J_+((\inf_{k \geq 1} T_k)^-)^{p} < \infty$
(argue in the same way as in the proof of Case 2 on p.\;\pageref{sas}).
\newline
\underline{Subcase 2a:} $\xi\geq 0$ a.s. Set $f(x):=J_+(x)^{p}$.
$f$ is absolutely continuous, in particular a.e.\ differentiable with derivative $f'$.
Therefore, it is sufficient to show that
\begin{equation*}
K ~:=~ \int_0^\infty f'(u)\Prob\{-\inf_{k \geq 1} T_k>u\} \, {\rm d}u ~<~ \infty.
\end{equation*}
Since
\begin{equation*}
\Prob \{-\inf_{k \geq 1}T_k > u\}
~\leq~ \sum_{k \geq 1} \Prob\{T_k\leq -u\}
~=~ \E \mathbb{U}_0(-u-\eta),
\end{equation*}
we have
\begin{equation*}
K ~=~ \E \int_0^{\eta^-} \!\! f'(u)\mathbb{U}_0(-u-\eta) \, {\rm d}u
~\leq~ \E f(\eta^-)\mathbb{U}_0(\eta^-)
~<~ \infty.
\end{equation*}
The assertion follows in view of the asymptotics \eqref{eq:asymp}
and the assumption $\E J_+(\eta^-)^{p+1}<\infty$.

\noindent
\underline{Subcase 2b:}\footnote{The reason for the separate treatment of Subcases 2a and 2b is as follows.
Assume that Ê$\Prob\{\xi<0\} > 0$. When $p \geq 1$, the argument given for Subcase 2a works as well since then,
due to the assumption $\E J_+(\xi^-)^{p+1} < \infty$, we also have $\mathbb{U}_0(y) < \infty$ for all $y$.
However, when $p \in (0,1)$ that argument fails which forces us to treat the case
$\Prob\{\xi<0\}>0$ separately as Subcase 2b.}
$\Prob\{\xi<0\}>0$.
Define the stopping times
\begin{equation*}
\tau^{*}_0:=0, \quad \tau^{*}_{n+1}:=\inf\{k>\tau^{*}_n: S_k > S_{\tau^{*}_n}\},
\quad n \in \N_0.
\end{equation*}
By assumption, $\lim_{n \to \infty} S_n=\infty$ a.s.
Hence, each $\tau^{*}_n$ is a.s.\ finite.
For $k \in \N_0$, define new random variables as follows:
\begin{equation*}
\widehat{\eta}_{k}:=
\min\{\eta_{\tau^{*}_{k-1}+1}, \xi_{\tau^{*}_{k-1}+1}+\eta_{\tau^{*}_{k-1}+2},\ldots,\, \xi_{\tau^{*}_{k-1}+1}+\ldots+\xi_{\tau^{*}_{k}-1}+\eta_{\tau^{*}_{k}}\};
\end{equation*}
\begin{equation*}
\widehat{\xi}_{k}:=\xi_{\tau^{*}_{k-1}+1}+\ldots+\xi_{\tau^{*}_{k}}.
\end{equation*}
The random vectors $(\widehat{\xi}_k, \widehat{\eta}_k)$, $k\in\N$,
are independent copies of the random vector
$(S_{\tau^{*}_1}, \min_{1 \leq k \leq \tau^{*}_1} T_k)$.
Denote by $(\widehat{T}_k)_{k\in\N}$ the perturbed random walk generated by
the vectors $(\widehat{\xi}_k,\widehat{\eta_k})$, $k \in \N$, \textit{i.e.},
\begin{equation*}
\widehat{T}_k:=\widehat{S}_{k-1}+\widehat{\eta}_{k}, \quad  k\in\N,
\end{equation*}
where
\begin{equation*}
\widehat{S}_0:=0, \quad \widehat{S}_k:=\widehat{\xi}_1+\ldots+\widehat{\xi}_k, \quad k\in\N.
\end{equation*}
Note that, by construction, $\widehat{S}_k>0$ for all $k \in \N$. Finally,
\begin{equation*}
\inf_{k \geq 1} \widehat{T}_k ~=~ \inf_{k \geq 1}T_k.
\end{equation*}
According to the already established Subcase 2a
it suffices to prove that
\begin{equation}    \label{eq:aux}
\E J_+(\widehat{\eta}^-) ^{p+1}= \E J_+ \Big(\Big(\min_{1 \leq k\leq \tau^{*}_1}T_k\Big)^-\Big)^{p+1} < \infty.
\end{equation}
To this end, obtain that, a.s.,
\begin{equation*}
\Big(\min_{1 \leq k \leq \tau^{*}_1} T_k \Big)^-
~\leq~ \Big| \min_{0 \leq k \leq \tau^{*}_1-1} S_k \Big| + \Big(\min_{1 \leq k \leq \tau^{*}_1}\eta_k \Big)^-
~\leq~ \Big| \min_{0 \leq k \leq \tau^{*}_1-1} S_k \Big| + \sum_{k=1}^{\tau^{*}_1}\eta_k^-.
\end{equation*}
Hence, using the monotonicity and subadditivity of $x \mapsto J_+(x)$
we conclude that a.s.
\begin{eqnarray*}
J_+\Big(\Big(\min_{1\leq k\leq \tau^{*}_1} T_k\Big)^-\Big)^{p+1}
& \leq &
\bigg(J_+ \Big( \Big| \min_{0 \leq k \leq \tau^{*}_1-1} S_k \Big| \Big) + \sum_{k=1}^{\tau^{*}_1}J_+(\eta_k^-)\bigg)^{p+1}   \\
& \leq &
2^p \bigg(J_+ \Big(\Big| \min_{0 \leq k \leq \tau^{*}_1-1} S_k \Big| \Big)^{p+1} + \bigg(\sum_{k=1}^{\tau^{*}_1}J_+(\eta_k^-)\bigg)^{p+1}\bigg)
\end{eqnarray*}
Using the already proved equivalence \eqref{eq:Erho*^p<infty}$\Leftrightarrow$\eqref{eq:criterion_Erho^p<infty},
$\E \rho^*(y)^p<\infty$ for some $y\geq 0$ implies that $\E (\tau^{*}_1)^{p+1}<\infty$ and $\E J_+(\xi^-)^{p+1}<\infty$.
Hence
\begin{equation}
\E J_+\Big(\Big | \min_{0 \leq k \leq \tau^{*}_1-1} S_k \Big | \Big)^{p+1} ~<~ 	\infty,
\end{equation}
by virtue of Lemma\;\ref{Lem:EJ^p+1xi^-<infty}.
Further, by \cite[Theorem\;5.2 on p.\;24]{Gut:09},
\begin{equation*}
\E \bigg(\sum_{k=1}^{\tau^{*}_1}J_+(\eta_k^-)\bigg)^{p+1}
\leq {\rm const} \, \E J_+(\eta^-)^{p+1} \E (\tau^{*}_1)^{p+1}<\infty,
\end{equation*}
and \eqref{eq:aux} follows. The proof is complete.
\end{proof}

\subsection{Proofs of the Results on Finiteness of Power Moments of $\tau$}	\label{subsec:proofs power moments tau}

\begin{proof}[Proof of Proposition\;\ref{Prop:Etau*^p<infty=>Etau^p<infty}]
When $\Prob\{\xi \geq 0\} = 1$, then $\E \tau(x)^p < \infty$ for all $x \geq 0$ by Theorem\;\ref{Thm:exponential_tau}.
Therefore, we restrict ourselves to the case $\Prob\{\xi < 0\} > 0$. Assume that $\E (\tau^*)^p < \infty$.
We use a technique similar to the technique used in the proof of Theorem\;\ref{Thm:exponential_tau}.
Pick $s \in \R$ such that $\gamma := \Prob\{\eta_{\tau^*} < \xi_{\tau^*}+s\} < 1$
and let $\sigma := \inf\{n \geq 1: \eta_{\tau_n^*} \geq \xi_{\tau_n}^*+s\}$. Then $\tau(s) \leq \tau_{\sigma}^*$ and
\begin{eqnarray*}
\E (\tau_{\sigma}^*)^p
& = &
\sum_{n \geq 1} \E \1_{\{\sigma=n\}} (\tau_n^*)^p	\\
& \leq &
\sum_{n \geq 1} \E \1_{\{\sigma=n\}} (n^{p-1} \vee 1) \sum_{k=1}^n (\tau_k^*-\tau_{k-1}^*)^p	\\
& \leq &
\E (\tau^*)^p \sum_{n \geq 1} (n^{p-1} \vee 1) \gamma^{n-2}
~<~ \infty.
\end{eqnarray*}
$\E \tau(x)^p < \infty$ for arbitrary $x \geq 0$ can now be concluded as in the proof Theorem\;\ref{Thm:exponential_tau} for exponential moments.
\end{proof}

\begin{proof}[Proof of Proposition\;\ref{Prop:Etau*=infty_and_Etau<infty}]
The first assertion follows from $\sigma(x) \geq \tau(x)$.
Concerning the second assertion, notice that $\E \sigma(x)^p$ is finite iff
\begin{equation*}
\sum_{k \geq 1} k^{p-1} \Prob\{\sigma(x) > k\}	~=~	\sum_{k \geq 1} k^{p-1} \prod_{j=1}^{k} \Prob\{\eta \leq x + c(j-1)\}	~<~	\infty.
\end{equation*}
Denote the $k$th summand in the series above by $a_k$, $k \in \N$.
By Raabe's test, the series converges if $\lim_{k \to \infty} k(\frac{a_k}{a_{k+1}} - 1) > 1$ and it diverges if the limit is $<1$.
Now
\begin{eqnarray*}
k \Big(\frac{a_k}{a_{k+1}} - 1\Big)
& = &
k \Bigg(\frac{k^{p-1} - (k+1)^{p-1} \Prob\{\eta \leq x + ck\}}{(k+1)^{p-1} \Prob\{\eta \leq x + ck\}}\Bigg)	\\
& \sim &
k (1 - (1+1/k)^{p-1} \Prob\{\eta \leq x + ck\})	\\
& = &
k (1 - (1 +(p-1)/k + o(1/k)) \Prob\{\eta \leq x + ck\})	\\
& = &
k\Prob\{\eta > x + ck\} - (p-1)\Prob\{\eta \leq x+ck\} + o(1)	\\
& \to &
s/c - (p-1)	\qquad	\text{as } k \to \infty.
\end{eqnarray*}
This limit is $> 1$ if $s > cp$ and it is $<1$ if $s < cp$.
\end{proof}

\footnotesize
\noindent	{\bf Acknowledgements}	\quad
The research of G.\;Alsmeyer was supported by DFG SFB 878 ``Geometry, Groups and Actions''.
The research of M.\;Meiners was partly supported by DFG-grant Me 3625/1-1 and DFG SFB 878 ``Geometry, Groups and Actions''.
A part of this study was done
while A.\;Iksanov was visiting M\"{u}nster in January/February and
May 2011. Grateful acknowledgment is made for financial support
and hospitality. Also supported by a grant awarded by the
President of Ukraine (project $\Phi$47/012) and partly
supported by a grant from Utrecht University, the Netherlands.
The authors thank an anonymous referee for a careful reading of the manuscript and helpful comments.
\normalsize

\section{Appendix: Auxiliary Results}	\label{sec:appendix}

\subsection{Auxiliary Results from Classical Random Walk Theory}	\label{subsec:aux_RW}

This section contains some facts from classical random walk theory that are either
reformulations or slight extensions of known results.
The first result is a combination of Theorems\;2.1 and 2.2 in \cite{IksMei:10a}.

\begin{Prop}   \label{Prop:IksMei:10}
For $a>0$, let $V^*_a(I) := \sum_{n \geq 0} e^{an} \Prob\{S_n \in I\}$, $I \subseteq \R$ Borel and $V^*_a(x) := V^*_a((-\infty,x])$, $x \in \R$. Further, let $R:=-\log \inf_{t \geq 0} \E e^{-t \xi}$.
\begin{itemize}
    \item[(a)]
    \begin{itemize}
    \item[(i)]
    Assume that $\Prob\{\xi\geq 0\}=1$ and let
    $\beta:=\Prob\{\xi=0\}\in [0,1)$. Then for $a>0$ the following
    conditions are equivalent:
    \begin{equation}	\label{eq:V^*_a(x)<infty}
    V^*_a(x) < \infty \text{ for some/all } x\geq 0;
    \end{equation}
    \begin{equation}\label{eq:0<V^*_a(I)<infty}
    0 < V^*_a(I) < \infty \text{ for some bounded interval } I \subseteq \R;
    \end{equation}
    \begin{equation}    \label{eq:a<-log(beta)}
    a<-\log \beta
    \end{equation}
    where $-\log \beta := \infty$ if $\beta = 0$.
    \item[(ii)]
    Assume that $\Prob\{\xi<0\}>0$.
    Then for $a>0$ condition \eqref{eq:V^*_a(x)<infty} (with $x\in\R$) is equivalent to
    \begin{equation}    \label{eq:a<R_varphi'<0}
    a < R
    \quad   \text{or} \quad
    a=R \ \text{ and } \ \E \xi
    e^{-\gamma_0\xi}>0,
    \end{equation}
    where $\gamma_0$ is the unique positive value defined by $\E e^{-\gamma_0\xi}= e^{-R}$.
    \end{itemize}
    \item[(b)]
    Whenever $V^*_a(x)$ is finite,
    \begin{equation*}
    0 < \liminf_{x \to \infty} e^{-\gamma x} V^*_a(x)\leq \limsup_{x\to\infty} e^{-\gamma x}V^*_a(x) < \infty.
    \end{equation*}
\end{itemize}
\end{Prop}
Part (a) of the Proposition\;contains more equivalent criteria for the finiteness of the exponential renewal function of a random walk than Theorem\;2.1 in \cite{AlsIks:09}. For this reason, we decided to include a proof.
\begin{proof}
We begin with part (a)(i) and assume that $\Prob\{\xi \geq 0\} = 1$. Then the equivalence between \eqref{eq:V^*_a(x)<infty} and \eqref{eq:a<-log(beta)} follows from \cite[Theorem\;2.1(b) and (c)]{IksMei:10a}. Moreover, the implication ``\eqref{eq:V^*_a(x)<infty}$\Rightarrow$\eqref{eq:0<V^*_a(I)<infty}'' is trivial.
It remains to prove that $0 < V^*_a(I) < \infty$ implies that $a<-\log \beta$. We will use contraposition and assume that $a \geq -\log \beta$, in particular, $\beta > 0$. Then let $I \subseteq [0,\infty)$ denote an arbitrary bounded interval with $V^*_a(I) > 0$. We have to show that $V^*_a(I)=\infty$.
To this end, notice that $V^*_a(I) > 0$ implies that $\Prob\{S_n \in I\} > 0$ for some $n \in \N$. Then, for any $k \geq 0$, we infer
\begin{equation*}
\Prob\{S_{n+k} \in I\} \geq \Prob\{S_n \in I, \xi_{n+1} = \ldots = \xi_{n+k} = 0\} = \Prob\{S_n \in I\} \beta^k.
\end{equation*}
In conclusion,
\begin{eqnarray*}
V^*_a(I)
& = &
\sum_{k \geq 0} e^{ak} \Prob\{S_k \in I\}
\geq
\sum_{k \geq 0} e^{a(n+k)} \Prob\{S_{n+k} \in I\}    \\
& \geq &
e^{an}  \Prob\{S_n \in I\} \sum_{k \geq 0} (e^a\beta)^k = \infty.
\end{eqnarray*}

Part (a)(ii) follows from Theorem\;2.1(a) in \cite{IksMei:10a}.

Part (b) follows from Theorem\;2.2 in \cite{IksMei:10a}.
\end{proof}

Lemma\;\ref{Lem:U_p_local_finiteness} will be used in the proof of Theorem\;\ref{Thm:power_N}.

\begin{Lemma} \label{Lem:U_p_local_finiteness}
Let $p>0$ and $I \subseteq \R$ be an open interval such that
$0 < \E \left( \sum_{n \geq 0} \1_{\{S_n \in I\}}\right)^p < \infty$.
Then $\E \left( \sum_{n \geq 0} \1_{\{S_n \in J\}}\right)^p <
\infty$ for any bounded interval $J\subseteq \R$.
In particular, $\E N^*(x)^p<\infty$ for some $x \in \R$ entails
$\E N^*(y)^p<\infty$ for every $y \in \R$.
\end{Lemma}
\begin{Rem}
In the case that $x \geq 0$ the second assertion was known from
\cite{Kesten+Maller:96}.
\end{Rem}
\begin{proof}
Let $I = (a,b)$ such that
$0 < \E \left( \sum_{n \geq 0} \1_{\{S_n \in I\}}\right)^p < \infty$.
We assume w.l.o.g.\ that $-\infty < a < b < \infty$.
We first show that
\begin{equation}    \label{eq:E(sum_|S_n|<eps)^p<infty}
\E \bigg(\sum_{n \geq 0} \1_{\{|S_n| < \varepsilon\}}\bigg)^p < \infty
\quad	\text{for some } \varepsilon > 0.
\end{equation}
Pick $\varepsilon > 0$ so small that $I_{\varepsilon} := (a+\varepsilon,b-\varepsilon)$
satisfies $\E \left( \sum_{n \geq 0} \1_{\{S_n \in I_{\varepsilon}\}}\right)^p > 0$.
Then $\Prob\{S_n \in I_{\varepsilon}\} > 0$ for some $n \in \N$. In particular,
$\Prob\{\tau^{*}(I_{\varepsilon}) < \infty\} > 0$, where
$\tau^{*}(I_{\varepsilon}) = \inf\{n \geq 0: S_n \in I_{\varepsilon}\}$.
Using the strong Markov property at $\tau^{*}(I_{\varepsilon})$, we get
\begin{eqnarray*}
\infty & > &
\E \bigg( \sum_{n \geq 0} \1_{\{S_n \in I\}}\bigg)^p   \\
& \geq &
\E \bigg( \1_{\{\tau^{*}(I_{\varepsilon})<\infty\}} \sum_{n \geq \tau^{*}(I_{\varepsilon})}
\1_{\{|S_n - S_{\tau^{*}(I_{\varepsilon})}| < \varepsilon\}}\bigg)^p   \\
& = &
\Prob\{\tau^{*}(I_{\varepsilon}) < \infty\} \, \E \bigg(\sum_{n \geq 0} \1_{\{|S_n| < \varepsilon\}}\bigg)^p.
\end{eqnarray*}
Hence, \eqref{eq:E(sum_|S_n|<eps)^p<infty} holds.
Now let $J$ be a non-empty bounded interval $\subseteq \R$, and
$J_1, \ldots, J_m$ open intervals of length at most $\varepsilon$
such that $J \subseteq J_1 \cup \ldots \cup J_m$.
Using the inequality
$(x_1+\ldots+x_m)^p \leq (m^{p-1}\vee 1)(x_1^p+\ldots+x_m^p)$, $x_j \geq 0$ for $j=1,\ldots,m$,
leads to
\begin{eqnarray*}
\E \bigg( \sum_{n \geq 0} \1_{\{S_n \in J\}}\bigg)^p
& \leq &
\E \bigg( \sum_{k=1}^m \sum_{n \geq 0} \1_{\{S_n \in J_k\}}\bigg)^p 	\\
& \leq &
(m^{p-1}\vee 1) \sum_{k=1}^m \E \bigg( \sum_{n \geq 0} \1_{\{S_n \in J_k\}}\bigg)^p.
\end{eqnarray*}
Therefore, it suffices to prove the result under the additional
assumption that the length of $J$ is at most $\varepsilon$.
Using the strong Markov property at
$\tau^{*}(J) := \inf\{n \geq 0: S_n \in J\}$ gives
\begin{eqnarray*}
\E \bigg( \sum_{n \geq 0} \1_{\{S_n \in J\}}\bigg)^p
& \leq &
\E \bigg( \1_{\{\tau^{*}(J)<\infty\}} \sum_{n \geq \tau^{*}(J)} \1_{\{|S_n -S_{\tau^{*}(J)}|<\varepsilon\}}\bigg)^p    \\
& = &
\Prob{\{\tau^{*}(J)<\infty\}} \E \bigg( \sum_{n \geq 0} \1_{\{|S_n|<\varepsilon\}}\bigg)^p < \infty.
\end{eqnarray*}
This proves the first assertion of the lemma. Concerning the second, assume that $\E N^*(x)^p<\infty$ for some $x \in \R$.
Then, for any $y > x$,
\begin{equation*}
\E N^*(y)^p
~\leq~ (2^{p-1} \vee 1) \bigg(\E N^*(x)^p+\E \bigg[\sum_{n \geq 0}\1_{\{x<S_n \leq y\}}\bigg]^p \bigg)
~<~ \infty,
\end{equation*}
where the last term is finite by the first part of the lemma.
\end{proof}

The following lemma summarizes properties of the functions $A_+$ and $J_+$ that are frequently used throughout the proofs.
These properties were known before and are stated here only for the reader's convenience.
Recall from \eqref{eq:A(x)} that 
\begin{equation*}
A_+(x) := \int_0^x \Prob\{\xi>y\} \, {\rm d}y = \E \min (\xi^+, x)
\quad	\text{and}	\quad	J_+(x):= \frac{x}{A_+(x)}
\end{equation*}
whenever $A_+(x) > 0$.
Further, recall that $\mathbb{U}_{p-1}(x) = \sum_{n \geq 1} n^{p-1} \Prob\{S_n \leq x\}$ and,
analogously, $\mathbb{U}_{p-1}^>(x) = \sum_{n \geq 1} n^{p-1} \Prob\{S_{\tau_n^*} \leq x\}$
where $\tau_n^*$ is the $n$th strictly ascending ladder index of the random walk $(S_n)_{n \geq 0}$.

\begin{Lemma}	\label{Lem:J_+}
Assume that $S_n \to \infty$ a.s. Then the following assertions are true:
\begin{itemize}
	\item[(a)]	$A_+(x) > 0$ for all $x > 0$; $A_+$ and $J_+$ are nondecreasing.
	\item[(b)]	$\lim_{x \to \infty} J_+(x) = \infty$.
	\item[(c)]	$J_+$ is subadditive, \textit{i.e.}, $J_+(x+y) \leq J_+(x)+J_+(y)$ for all $x, y \geq 0$. In particular,
				$J_+(x+y) \sim J_+(x)$ as $x \to \infty$ for any $y \in \R$.
	\item[(d)]	For any $p > 0$ if $\mathbb{U}_{p-1}(0)<\infty$, then
				\begin{equation}	\label{eq:asymp}
				\mathbb{U}_{p-1}(x) \asymp \mathbb{U}_{p-1}^>(x) \asymp J_+(x)^p
				\end{equation}
				as $x \to \infty$. Moreover, with $W(\cdot)$ denoting either $\tau^*(\cdot)$, $N^*(\cdot)$ or $\rho^*(\cdot)$,
				then $\E W(x)^p \asymp J_+(x)^p$ whenever $\E W(0)^p < \infty$.
\end{itemize}
\end{Lemma}
\begin{proof}
(a) Since $S_n \to \infty$ a.s.\ is assumed, $\Prob\{\xi^+ > 0\}=\Prob\{\xi > 0\} > 0$ and therefore $\Prob\{\xi>y\}>0$ in a right neighborhood of $0$.
The monotonicity of $A_+$ follows from its definition. The monotonicity of $J_+$ and assertion (b) follow from the following representation
\begin{equation*}
J_+(x) = \left(\int_0^1 \Prob\{\xi>xy\} \, {\rm d}y \right)^{-1} \!\!\!,
\quad	x > 0.
\end{equation*}
Regarding (c) notice that the subadditivity of $J_+$ follows from the monotonicity of $A_+$. $J_+(x+y) \sim J_+(x)$ as $x \to \infty$ immediately follows from the subadditivity of $J_+$ together with (b).

\noindent
(d) follows from equations \cite[Theorems\;2.1 and 2.2, Eq.\;(4.5)]{Kesten+Maller:96} and one of the displayed formulas on p.\;28 of the cited reference.
\end{proof}

\begin{Lemma}       \label{Lem:EJ^p+1xi^-<infty}
Let $p>0$ and assume $\lim_{k \to \infty} S_k=+\infty$ a.s.
Then the following assertions are equivalent:
\begin{align}
\E J_+ \Big(\Big| \min_{0 \leq k \leq \tau^*-1} S_k \Big|\Big)^{p+1}  ~&<~ \infty;	\\
\E J_+ \Big(\Big| \inf_{k \geq 0} S_k \Big| \Big)^p  ~&<~ \infty;	\\
\E J_+(\xi^-)^{p+1} ~&<~ \infty
\end{align}
where $\tau^*:=\inf\{k \in \N: S_k>0\}$.
\end{Lemma}

Lemma\;\ref{Lem:EJ^p+1xi^-<infty} has several predecessors, \textit{e.g.}\ 
\cite[Theorem\;1]{Janson:86}, \cite[Theorem\;3]{Alsmeyer:92}, \cite[Proposition\;4.1]{Kesten+Maller:96},
\cite[Lemma\;3.5]{AlsIks:09}.
Even though Lemma\;\ref{Lem:EJ^p+1xi^-<infty} does not follow directly from either of these results,
the proofs given in \cite{Janson:86} and \cite{Kesten+Maller:96} can be
adopted to treat the present case after the observation
that the function $x \mapsto J_+(x)$ is nondecreasing and subadditive.
Therefore, we omit a proof.

\subsection{Elementary Facts}		\label{subsec:aux_elementary}

\begin{Lemma}	\label{Lem:Gut's_estimate_modified}
Let $1 \leq p = n+ \delta$ with $n \in \N_0$ and $\delta \in (0,1]$. Then, for any $x,y \geq 0$,
\begin{equation}	\label{eq:Gut's_estimate_modified}
(x+y)^p ~\leq~ x^p+y^p+p2^{p-1}(xy^{p-1} + x^ny^{\delta}).
\end{equation}
\end{Lemma}
This estimate is a variant of an estimate we have learned from \cite{Gut:74}.
For the reader's convenience, we include a brief proof which is a slight modification of the argument given in the cited reference.
\begin{proof}
For any $0 \leq r \leq 1$, we have $(1+r)^p = 1+p \int_0^r (1+t)^{p-1} {\mathrm d}t$.
By the mean value theorem for integrals, for some $\gamma \in (0,r)$,
\begin{equation}
(1\!+\!r)^p ~=~ 1+p r (1\!+\!\gamma)^{p-1} ~\leq~ 1+ p2^{p-1}r ~\leq~ 1+p2^{p-1}r^{\delta},	\label{eq:mean_value_estimate}
\end{equation}
where in the last step we have used that $0 \leq r \leq 1$. Now let $x, y \geq 0$.
When $x \leq y$, use the first estimate in \eqref{eq:mean_value_estimate} to get
$(x+y)^p \leq y^p +p2^{p-1}xy^{p-1}$. When $y \leq x$ use the second estimate in \eqref{eq:mean_value_estimate} to infer 
$(x+y)^p \leq x^p + p2^{p-1} x^n y^{\delta}$. Thus, in any case, \eqref{eq:Gut's_estimate_modified} holds.
\end{proof}

The next auxiliary result is an elementary consequence of a version of the summation by parts formula.

\begin{Lemma}   \label{Lem:summation_by_parts}
Let $b_n \geq 0$ for all $n \in \N$ and $p>1$. Then
\begin{equation*}
\sum_{n \geq 1} n^{p-1} \sum_{k=n}^\infty b_k ~<~ \infty
\quad	\text{iff}	\quad
\sum_{n \geq 1} n^p b_n ~<~ \infty.
\end{equation*}
\end{Lemma}
\begin{proof}
For arbitrary $m \in \N$,
\begin{equation}    \label{eq:summation_by_parts}
\sum_{n=1}^m n^{p-1}\sum_{k=n}^\infty b_k
~=~	\sum_{n=1}^m n^{p-1} \sum_{k=m}^\infty b_k + \sum_{k=1}^{m-1} b_k \sum_{n=1}^k n^{p-1}.
\end{equation}
In particular,
\begin{equation*}
\sum_{n=1}^m n^{p-1} \sum_{k=n}^\infty b_k ~\geq~ \sum_{n=1}^{m-1} b_n\sum_{k=1}^n k^{p-1}.
\end{equation*}
Consequently, if $\sum_{n \geq 1} n^{p-1}\sum_{k=n}^\infty b_k$ converges, then so does
$\sum_{n \geq 1} b_n\sum_{k=1}^n k^{p-1}$ and thus also $\sum_{n \geq 1} b_n n^p$.
Conversely, if the series $\sum_{n \geq 1} n^p b_n$ converges, then $\sum_{n \geq 1} b_n\sum_{k=1}^n k^{p-1}$ converges
and, in particular,
\begin{equation*}
\lim_{m\to\infty} \sum_{n=m}^{\infty} b_n \sum_{k=1}^n k^{p-1} ~=~ 0.
\end{equation*}
Further,
\begin{equation*}
0 ~\leq~ \sum_{k=1}^m k^{p-1} \sum_{n=m}^\infty b_n ~\leq~ \sum_{n=m}^{\infty} b_n\sum_{k=1}^n k^{p-1} ~\to~ 0
\quad   \text{as } m \to \infty.
\end{equation*}
Letting $m$ tend to $\infty$ in \eqref{eq:summation_by_parts}, we conclude that $\sum_{n \geq 1} n^{p-1}\sum_{k=n}^\infty b_k$ converges.
\end{proof}

\footnotesize
\def\cprime{$'$}


\begin{thebibliography}{10}

\bibitem{Alsmeyer:91}
G.~Alsmeyer.
\newblock {\em Erneuerungstheorie}.
\newblock {T}eubner {S}kripten zur {M}athematischen {S}tochastik. [Teubner
  texts on Mathematical Stochastics]. B. G. {T}eubner, Stuttgart (1991).
\newblock Analyse stochastischer {R}egenerationsschemata. [Analysis of
  stochastic regeneration schemes].

\bibitem{Alsmeyer:92}
G.~Alsmeyer.
\newblock On generalized renewal measures and certain first passage times.
\newblock {\em Ann. Probab.}, 20(3):1229--1247, 1992.

\bibitem{AlsIks:09}
G.~Alsmeyer and A.~Iksanov.
\newblock A log-type moment result for perpetuities and its application to
  martingales in supercritical branching random walks.
\newblock {\em Electron. J. Probab.}, 14(10):289--312, 2009.

\bibitem{AlsIksRoe:09}
G.~Alsmeyer, A.~Iksanov, and U.~R{\"o}sler.
\newblock On distributional properties of perpetuities.
\newblock {\em J. Theoret. Probab.}, 22(3):666--682, 2009.

\bibitem{AraGlynn:06}
V.~F. Araman and P.~W. Glynn.
\newblock Tail asymptotics for the maximum of perturbed random walk.
\newblock {\em Ann. Appl. Probab.}, 16(3):1411--1431, 2006.

\bibitem{Asmussen:03}
S.~Asmussen.
\newblock {\em Applied Probability and Queues}.
\newblock $2^{\text{nd}}$ edn., Springer, New York, , 2003.

\bibitem{BurkDG:72}
D.~L. Burkholder, B.~J. Davis, and R.~F. Gundy.
\newblock Integral inequalities for convex functions of operators on
  martingales.
\newblock In: {\em Proceedings of the {S}ixth {B}erkeley {S}ymposium on
  {M}athematical {S}tatistics and {P}robability ({U}niv. {C}alifornia,
  {B}erkeley, CA, 1970/1971), vol. {II}: {P}robability theory}, pages
  223--240, Berkeley, CA, 1972. Univ. California Press.

\bibitem{Burkholder+Gundy:70}
D.~L. Burkholder and R.~F. Gundy.
\newblock Extrapolation and interpolation of quasi-linear operators on
  martingales.
\newblock {\em Acta Math.}, 124:249--304, 1970.

\bibitem{Chow:94}
Y.~S. Chow.
\newblock On the moments of ladder epochs for driftless random walks.
\newblock {\em J. Appl. Probab.}, 31A:201--205, 1994.
\newblock Studies in applied probability.

\bibitem{Chow+Robbins:91}
Y.~S. Chow, H.~Robbins, and D.~Siegmund.
\newblock {\em The theory of Optimal Stopping}.
\newblock Dover Publications Inc., New York, 1991.
\newblock Corrected reprint of the 1971 original.

\bibitem{DoneyOBrian:91}
R.~A. Doney and G.~L. O'Brien.
\newblock Loud shot noise.
\newblock {\em Ann. Appl. Probab.}, 1(1):88--103, 1991.

\bibitem{Erickson:73}
K.~B. Erickson.
\newblock The strong law of large numbers when the mean is undefined.
\newblock {\em Trans. Am. Math. Soc.}, 185:371--381, 1973.

\bibitem{FillHuber:10}
J.~A. Fill and M.~L. Huber.
\newblock Perfect simulation of {V}ervaat perpetuities.
\newblock {\em Electron. J. Probab.}, 15:96--109, 2010.

\bibitem{GneIksMar:10a}
A.~Gnedin, A.~Iksanov, and A.~Marynych.
\newblock The {B}ernoulli sieve: an overview.
\newblock Discrete Math. Theor. Comput. Sci. Proc., AI, pages 329--342. Assoc.
  Discrete Math. Theor. Comput. Sci. Nancy, 2010.

\bibitem{GneIksMar:10b}
A.~Gnedin, A.~Iksanov, and A.~Marynych.
\newblock Limit theorems for the number of occupied boxes in the {B}ernoulli
  sieve.
\newblock {\em Theory Stoch. Process.}, 16(32)(2):44--57, 2010.

\bibitem{Goldie:91}
C.~M. Goldie.
\newblock Implicit renewal theory and tails of solutions of random equations.
\newblock {\em Ann. Appl. Probab.}, 1(1):126--166, 1991.

\bibitem{GolMal:00}
C.~M. Goldie and R.~A. Maller.
\newblock Stability of perpetuities.
\newblock {\em Ann. Probab.}, 28(3):1195--1218, 2000.

\bibitem{Gut:74}
A.~Gut.
\newblock On the moments and limit distributions of some first passage times.
\newblock {\em Ann. Probab.}, 2:277--308, 1974.

\bibitem{Gut:09}
A.~Gut.
\newblock {\em Stopped Random Walks. Limit Theorems and Applications.}
\newblock Springer Series in Operations Research and Financial Engineering.
  Springer, New York, $2^{\text{nd}}$ edn., 2009.

\bibitem{HaoTangWei:09}
X.~Hao, Q.~Tang, and L.~Wei.
\newblock On the maximum exceedance of a sequence of random variables over a
  renewal threshold.
\newblock {\em J. Appl. Probab.}, 46(2):559--570, 2009.

\bibitem{Hitc:88}
P.~Hitczenko.
\newblock Comparison of moments for tangent sequences of random variables.
\newblock {\em Probab. Theory Related Fields}, 78(2):223--230, 1988.

\bibitem{Hitc:10}
P.~Hitczenko.
\newblock On tails of perpetuities.
\newblock {\em J. Appl. Probab.}, 47(4):1191--1194, 2010.

\bibitem{HitcWes:09}
P.~Hitczenko and J.~Weso{\l}owski.
\newblock Perpetuities with thin tails revisited.
\newblock {\em Ann. Appl. Probab.}, 19(6):2080--2101, 2009.
\newblock Corrigendum in 20(3):1177 (2010).

\bibitem{HitcWes:11}
P.~Hitczenko and J.~Weso{\l}owski.
\newblock Renorming divergent perpetuities.
\newblock {\em Bernoulli}, 17(3):880--894, 2011.

\bibitem{Iks:07}
A.~Iksanov.
\newblock {\em Fixed points of inhomogeneous smoothing transforms}.
\newblock Unpublished manuscript, 2007.
\newblock Thesis (habilitation) --National T. Shevchenko University of Kiev.

\bibitem{IksMei:10b}
A.~Iksanov and M.~Meiners.
\newblock Exponential moments of first passage times and related quantities for
  random walks.
\newblock {\em Electron. Commun. Probab.}, 15:365--375, 2010.

\bibitem{IksMei:10a}
A.~Iksanov and M.~Meiners.
\newblock Exponential rate of almost-sure convergence of intrinsic martingales
  in supercritical branching random walks.
\newblock {\em J. Appl. Probab.}, 47(2):513--525, 2010.

\bibitem{Iks:01}
A.~M. Iksanov.
\newblock Parameter estimation for the radioactive contamination process.
\newblock {\em Studia Sci. Math. Hungar.}, 37(3-4):237--258, 2001.

\bibitem{Iks:12}
A.~M. Iksanov.
\newblock Functional limit theorems for renewal shot noise processes, 2012.
\newblock Preprint available at \texttt{arxiv.org/abs/1202.1950}.

\bibitem{Janson:86}
S.~Janson.
\newblock Moments for first-passage and last-exit times, the minimum, and
  related quantities for random walks with positive drift.
\newblock {\em Adv. Appl. Probab.}, 18(4):865--879, 1986.

\bibitem{Kesten+Maller:96}
H.~Kesten and R.~A. Maller.
\newblock Two renewal theorems for general random walks tending to infinity.
\newblock {\em Probab. Theory Related Fields}, 106(1):1--38, 1996.

\bibitem{Konstantopoulos+Lin:98}
T.~Konstantopoulos and S.-J. Lin.
\newblock Macroscopic models for long-range dependent network traffic.
\newblock {\em Queueing Syst. Theory Appl.}, 28(1-3):215--243, 1998.

\bibitem{LaiSieg:77}
T.~L. Lai and D.~Siegmund.
\newblock A nonlinear renewal theory with applications to sequential analysis.
  {I}.
\newblock {\em Ann. Statist.}, 5(5):946--954, 1977.

\bibitem{LaiSieg:79}
T.~L. Lai and D.~Siegmund.
\newblock A nonlinear renewal theory with applications to sequential analysis.
  {II}.
\newblock {\em Ann. Statist.}, 7(1):60--76, 1979.

\bibitem{Lebedev:02}
A.~V. Lebedev.
\newblock Extremes of subexponential shot noise.
\newblock {\em Math. Notes}, 71(1--2):206--210, 2002.

\bibitem{McCormick:97}
W.~P. McCormick.
\newblock Extremes for shot noise processes with heavy tailed amplitudes.
\newblock {\em J. Appl. Probab.}, 34(3):643--656, 1997.

\bibitem{Mikosch+Resnick:06}
T.~Mikosch and S.~Resnick.
\newblock Activity rates with very heavy tails.
\newblock {\em Stoch. Process. Appl.}, 116(2):131--155, 2006.

\bibitem{PalZwart:07}
Z.~Palmowski and B.~Zwart.
\newblock Tail asymptotics of the supremum of a regenerative process.
\newblock {\em J.~Appl.~Probab.}, 44(2):349--365, 2007.

\bibitem{PalZwart:10}
Z.~Palmowski and B.~Zwart.
\newblock On perturbed random walks.
\newblock {\em J. Appl. Probab.}, 47(4):1203--1204, 2010.

\bibitem{Rice:77}
J.~Rice.
\newblock On generalized shot noise.
\newblock {\em Adv.~Appl.~Probab.}, 9(3):553--565, 1977.

\bibitem{Schottky:18}
W.~Schottky.
\newblock Spontaneous current fluctuations in electron streams.
\newblock {\em Ann. Phys.}, 57:541--567, 1918.

\bibitem{Takacs:56a}
L.~Tak{\'a}cs.
\newblock On secondary stochastic processes generated by recurrent processes.
\newblock {\em Acta Math. Acad. Sci. Hungar.}, 7:17--29, 1956.

\bibitem{Uchiyama:11}
K.~Uchiyama.
\newblock A note on summability of ladder heights and the distributions of
  ladder epochs for random walk.
\newblock {\em Stoch. Processes Appl.}, 121(9):1938--1961, 2011.

\bibitem{Woodroofe:82}
M.~Woodroofe.
\newblock {\em Nonlinear renewal theory in sequential analysis}, volume~39 of
  {\em CBMS-NSF Regional Conference Series in Applied Mathematics}.
\newblock Society for Industrial and Applied Mathematics (SIAM), Philadelphia, PA, 1982.

\end{thebibliography}
\end{document}